\documentclass[11pt]{article}
\usepackage[left=1in, right=1in, top=1in]{geometry}
\usepackage[toc,page]{appendix}
\usepackage[utf8]{inputenc}
\usepackage{amsfonts}
\usepackage{mathrsfs}
\usepackage[english]{babel}
\usepackage{amsmath}
\usepackage{relsize}
\usepackage{lipsum}
\usepackage{xcolor}
\usepackage[english]{babel} 
\usepackage{amsthm}
\usepackage{tikz-cd}
\usepackage{extpfeil}
\usepackage{mathtools}
\usepackage{csquotes}
\usepackage{color}
\usepackage{soul} 
\usepackage{amssymb}
\usepackage[normalem]{ulem}
\usepackage[nottoc,numbib]{tocbibind} 
\usepackage{authblk}
\usepackage{amscd}
\usepackage{enumerate}
\usepackage[colorlinks, citecolor=blue]{hyperref} 
\usepackage{bbold}

\newcommand{\C}{\mathbb{C}}

\newcommand{\Z}{\mathbb{Z}}
\newcommand{\delp}[1]{\frac{\partial}{\partial #1}}
\newcommand{\mc}[1]{\mathcal{#1}}
\newcommand{\pr}{\mathbb{P}^1}
\newcommand{\on}[1]{\operatorname{#1}}
\newcommand{\osp}{\operatorname{OSp}(1|2, \C)}
\newcommand{\posp}{\operatorname{POSp}(1|2, \C)}

\newcommand{\ov}[1]{\overline{#1}}

\newcommand{\hr}{\hookrightarrow}

\newcommand{\nr}{n_R} 
\newcommand{\sbs}{T_{\bos}}

\newcommand{\tbf}[1]{\textbf{#1}}

\newcommand{\ns}{n_{\on{NS}}}
\newcommand{\spin}{\mc{S} \mc{M}_g}

\newcommand{\mnr}{\mathfrak{M}_{0,\nr}}

\newcommand{\pp}{\mathbb{P}}

\newcommand{\Isom}{\un{\on{Isom}}}

\newcommand{\lgr}{\xrightarrow{\sim}}
\newcommand{\lr}{\longrightarrow}
\newcommand{\deform}{\on{Def}}
\newcommand{\prn}{\on{pr}_1}
\newcommand{\prt}{\on{pr}_2}

\newcommand{\delpz}{\frac{\partial}{\partial z}}
\newcommand{\delpzeta}{\frac{\partial}{\partial \zeta}}

\newextarrow{\xbigtoto}{{20}{20}{20}{20}}
   {\bigRelbar\bigRelbar{\bigtwoarrowsleft\rightarrow\rightarrow}}

\newtheorem{theorem}{Theorem}[section]
\newtheorem{lemma}[theorem]{Lemma}
\newtheorem{proposition}[theorem]{Proposition}
\newtheorem{corollary}[theorem]{Corollary}
\newtheorem{definition}[theorem]{Definition}

\theoremstyle{remark}
\newtheorem{remark}[theorem]{Remark}

\newtheorem{example}[theorem]{Example}

\DeclareMathOperator{\Aut}{Aut}
\DeclareMathOperator{\Ber}{Ber}
\DeclareMathOperator{\bos}{bos}
\DeclareMathOperator{\Conf}{Conf}

\DeclareMathOperator{\Hom}{Hom}
\DeclareMathOperator{\id}{id}

\DeclareMathOperator{\Proj}{Proj}

\DeclareMathOperator{\Spec}{Spec}

\DeclareRobustCommand\un[1]{\underline{#1}}

\newcommand{\wo}{\mathbb{W}\mathbb{P}^{1|1}(1,1\, |\, 1)} \newcommand{\wb}{\mathbb{W}\mathbb{P}^{1|1}(1,1\, |\, 1-\nr/2)} 
\newcommand{\wn}{\mathbb{W}\mathbb{P}^{1|1}(1,1\, |\, 1-n/2)} 
\newcommand{\wm}{\mathbb{W}\mathbb{P}^{1|1}(1,1\, |\, m)} 
 
\newcommand{\sS}{\on{SupSch}}

\newcommand{\wy}{\mathbb{W} \mathbb{P}}
\newcommand{\awy}{\Aut(\wy)}

\title{The supermoduli space of genus zero super Riemann surfaces\\ with Ramond punctures} 
\author{Nadia Ott\footnote{\emph{Email address}: \href{mailto:ottnadia@sas.upenn.edu}{ottnadia@sas.upenn.edu}\\ \indent Department of Mathematics, University of Pennsylvania, Philadelphia, PA 19104, USA
}
\hspace{1mm} and Alexander A. Voronov\footnote{\emph{Email address}: \href{mailto:voronov@umn.edu}{voronov@umn.edu}\\ \indent School of Mathematics, University of Minnesota, Minneapolis, MN 55455, USA, and \\ \indent Kavli IPMU (WPI), UTIAS, University of Tokyo, Kashiwa, Chiba 277-8583, Japan}}

\date{}

\begin{document}

\maketitle

\begin{abstract}
We give an explicit quotient description of the $(\nr-3\, | \,
  \nr/2-2)$-dimensional supermoduli space $\mnr$ of genus zero super Riemann surfaces with $\nr \ge 4$ Ramond punctures and prove that it is a Deligne-Mumford superstack. We also give an explicit quotient description of the supermoduli stack of genus zero supercurves with no SUSY structure.
\end{abstract}

\tableofcontents

\section{Introduction}

Super Riemann surfaces, also
known as supersymmetric (SUSY) curves, and their moduli spaces, called \emph{supermoduli spaces}, are supersymmetric generalizations of the
classical notions of Riemann surfaces. Super Riemann surfaces arose in
the 1980's as the worldsheets of propagating superstrings, and the
Feynman path integrals of superstring theory reduced to integrals over
the supermoduli spaces to produce such important quantities as the
partition function and scattering amplitudes (correlators).

Since the
2000's there has been a strong revival of the interest to these
topics, given a 2002 breakthrough \cite{dhoker-phong} in computing these quantities by D'Hoker and Phong in genus 2, a series of important papers \cite{witten2012notes, witten:integration, witten:perturbative, witten:superstring} of Witten
developing the theory further, reaching genus 3 and emphasizing complex analytic and
algebro-geometric aspects, and a 2015 Supermoduli Workshop at the
Simons Center for Geometry and Physics in Stony Brook, New York. These
days, papers presenting in-depth study of questions that seemed to be
out of reach in the 80s pop up now and then in print and on the arXiv:
papers of Codogni \cite{codogni}, Codogni and Viviani \cite{codogni2017moduli}, Diroff \cite{dan}, Dom\'{i}nguez
P\'{e}rez, Hern\'{a}ndez Ruip\'{e}rez, and Sancho de Salas \cite{perez1997global}, Donagi and
Witten \cite{donagi2015supermoduli}, Felder, Kazhdan, and Polishchuk \cite{felder2019regularity,felder2020moduli}, Fioresi and her coauthors \cite{carmeli2011mathematical,carmeli-fioresi,fioresi2014susy,fioresi2015projective,fioresi-kwok},
Maxwell \cite{maxwell2020super}, Moosavian and Zhou \cite{moosavian2019existence}, Rothstein and Rabin \cite{rabin-rothstein-D,rothstein-rabin}, and Th.~Th. Voronov \cite{voronov-t}, to name a few.

A super Riemann surface is described by the data $(X, \mc{D}, i)$, where $X$ is a a compact complex supermanifold of dimension $(1|1)$, and $\mc{D}$ is a rank $(0|1)$ vector bundle on $X$ with an isomorphism 
\[ i: \mc{D}^{\otimes 2} \overset{[ \ , \ ]}{\lr} \mc{T}_X/\mc{D} \]
given by the \emph{supercommutator} $[ \ , \ ]$. The data $(\mc{D},i)$ is called a \emph{SUSY or superconformal structure}. 

Despite this slightly exotic definition, super Riemann surfaces and their moduli have mathematical properties quite similar to those of classical Riemann surfaces or algebraic curves. For instance, prime divisors on a super Riemann surface, given as sub-supermanifolds of dimension $(0|1)$, are in one-to-one correspondence with points. In fact, if 
one studies families parameterized only by commuting (as opposed to anti-commuting) variables, then super Riemann surfaces are nothing other than spin curves. Another feature is that 
super Riemann surfaces have unobstructed deformation theory and hence are expected to have smooth moduli. Our interest here is in the moduli problem, which we will study from
the point of view of algebraic geometry.  The algebro-geometric approach to supermoduli theory was initiated by Deligne in a famous letter \cite{deligneletter} to 
Yu.~I. Manin, where the existence of a compactified moduli of stable super Riemann surfaces
was sketched, appealing to the (assumed) generalization of Schlessinger's conditions for pro-representability and Artin's existence theorems to supergeometry.\footnote{After the first version of the present article was posted on the arXiv, a detailed argument for the existence of the moduli space of stable super Riemann surfaces was given in \cite{moosavian2019existence}, along
the lines of Deligne's letter but without explicit use of Artin's theorems.}  
Works on supermoduli from other perspectives include \cite{lebrun1988moduli, perez1997global, codogni2017moduli}.

We are interested specifically in genus zero super Riemann surfaces, and in giving a construction of the supermoduli space by an explicit quotient presentation (rather than in an
abstract existence argument).

The desire to have a concrete presentation was expressed by E.~Witten in a letter to A.~S. Schwarz \cite{witten-to-schwarz}. As with ordinary curves, genus zero super Riemann surfaces present a certain challenge, as they have an infinitesimal group of automorphisms, but, as expected, if sufficiently many punctures are present, the moduli problem becomes 
representable by a Deligne-Mumford superstack.  (These punctures also appear naturally in the compactification of the supermoduli spaces of higher genus as well as superstring theory.)  In fact, there are 
two kinds of punctures on a super Riemann surfaces.

One kind, the Neveu-Schwarz punctures, are entirely analogous to marked points on ordinary Riemann surfaces, and their
moduli are straightforward to understand.
The moduli space $\mathfrak{M}_{0,\ns}$ of genus zero super Riemann surfaces with $\ns \ge 3$ Neveu-Schwarz punctures may be described just like in the classical case:
\[
\mathfrak{M}_{0,\ns} \cong \Conf (\C \pp^{1|1}, \ns)/\posp,
\]
where $\Conf (\C \pp^{1|1}, \ns)$ denotes the configuration space of $\ns$ distinct, labeled points \footnote{We use the word ``point" to mean a map $\on{Spec} \mathbb{C} \to \mathbb{C} \mathbb{P}^{1|1}$. More generally, given a family of superschemes $\pi: X \to T$, a $T$-point of $X$ is a section $s: T \to X$ such that $\pi \circ s=\on{id}_T$. } on the complex projective superspace $\C \pp^{1|1}$ and $\posp = \osp/\{\pm \id\}$ stands for the projective orthosymplectic supergroup, which is also the supergroup of  automorphisms of $\C \pp^{1|1}$ preserving the standard SUSY structure on it, cf.\ \cite{witten2012notes}. Note that the action of the group is free up to stabilizer $\Z_2$ and therefore $\dim \mathfrak{M}_{0,\ns} = \dim \Conf (\C \pp^{1|1}, \ns) - \dim \osp = (\ns-3, \ns -2)$.

By contrast, the Ramond punctures are subtler.\footnote{In a letter \cite{witten-to-schwarz} to Schwarz, Witten wrote: ``I am sorry to say this, but even in genus zero I have had trouble explicitly describing to my satisfaction the moduli space of a genus zero curve with $2k$ Ramond punctures.''}
To explain what these are (following \cite{witten2012notes}), first recall that 
locally  on a super Riemann surface $(X, \mc{D}, i)$, there exist coordinates $(z\, | \, \zeta)$, called \emph{superconformal coordinates}, such that $\mc{D}$ is generated  
by the vector field
\[ D=\delp{\zeta} + \zeta \delp{z}  \] from which we compute that $\frac12 [D, D] = \delp{z}$.

Now consider a subbundle  $\mc{D}$ of the tangent bundle locally generated  by the vector field
\[ D_{\zeta} = \delp{\zeta} + h(z) \zeta \delp{z}\]
where $h(z)$ is some degree-$\nr$ polynomial. Then $\frac12 [D, D_{\theta}] = h(z) \delp{z}$ and $D_{\zeta}$ is integrable along $h(z)= 0$. 

A \emph{super Riemann surface with $\nr$ Ramond punctures} is described by four pieces of data $(X, \mc{R}, \mc{D},i)$; where $X$ is as before but where we now allow $\mc{D}$ to be integrable exactly along a codimension-$(1|0)$ sub-supermanifold with $\nr$ components, \emph{i.e}., a \emph{divisor} $\mc{R}$ of degree $\nr$, in the sense that we have an isomorphism
\[ i: \mc{D}^{\otimes 2} \overset{[ \ , \ ]}{\lr} \mc{T}_X/\mc{D}(-\mc{R})\] 

We call $(\mc{D},i)$ a SUSY structure, though note that
$(X, \mc{R}, \mc{D}, i)$ is \textbf{not} a super Riemann surface in the sense our previous definition, since $\mc{D}$ is not everywhere nonintegrable. We term the components of $\mc{R}$  \emph{Ramond punctures}, though be aware that a codimension-$(1|0)$ sub-supermanifold on a supercurve is \tbf{not} a puncture, which
is a codimension-$(1|1)$ sub-supermanifold.

However, there exists no dictionary identifying codimension-$(1|0)$ divisors and punctures, when we are talking about Ramond punctures. It is the maximal non-integrability condition on $\mc{D}$ which implies such a duality for Neveu-Schwarz punctures.

Let us now describe the basic idea behind the construction of the moduli space of genus zero super Riemann surfaces with $n_R \ge 4$ Ramond punctures, $\mnr$.  Henceforth, we assume that all NS and Ramond punctures are labeled. 

It is useful to take the ``dualized'' definition of a SUSY structure, essentially a rank-$(1|0)$ subbundle $\mc{L}$ of the cotangent bundle meeting certain non-integrability conditions, see Definition \ref{SRSdef}.
We will\footnote{not following any standard convention} call $\mc{L}$ the \emph{SUSY line bundle}. 
In \cite{witten2012notes}, Witten gives a complete characterization of a genus zero SUSY curve with Ramond punctures over a point.
Over  a point, $\on{Spec} k$, any supercurve supporting a SUSY structure  with $\nr$ Ramond punctures is isomorphic to the \emph{weighted projective superspace} $\wb$ (henceforth $\wy$) and the SUSY line bundle
 is isomorphic to $\mc{O}_{\wy}(-2)$. A SUSY structure on $\wy$ can then be specified by giving a global section of $\mc{H}om(\mc{O}_{\wy}(-2), \Omega) \cong \Omega_{\wy}^1(2)$ satisfying the maximal non-integrability condition. This assignment is unique up to an invertible factor.

We first extend the characterization in \cite{witten2012notes} to families of super Riemann surfaces over ordinary, purely even schemes. It turns out that, unlike the standard projective superspace $\mathbb{P}^{1|1}$, the weighted projective superspace $\wb$ is not rigid but has purely odd moduli. This fact makes the supermoduli problem accessible by super deformation theory.

Globalizing the facts that (1) a genus zero super Riemann surface over a point is $\wy$ and (2) $\wy$ has no even deformations, we show in Proposition \ref{wittenclassifying} that,
given a family of genus zero $\nr$-punctured super Riemann surfaces over a purely even base, $\nr \ge 4$, the underlying family of supercurves is \'etale locally trivial.  

In Section~\ref{Supercurve}, we study the first-order deformations  of genus zero supercurves, showing in particular that they are \emph{not} (in general) rigid.   
This is quite in contrast to their bosonic counterpart $\pr$ as well as to the  projective superspace $\mathbb{P}^{1|1}$ (the supercurve underlying a SUSY curve when $\nr = 0$).  Later in Section \ref{deformations}, we prove that $\wy$ has a universal deformation space (Theorem \ref{versal}), 
denoted by $S$, and give an explicit construction of the universal deformation, denoted by $Z$, of $\wy$, in two ways.  The first way: $Z$ is glued from two copies
of $\mathbb{A}^{1|1}_S$, using the (constant in $S$) $z \to 1/w$ gluing for the even coordinate and a (varying in $S$) gluing for the odd coordinate \eqref{gluingformula}.  
The second way: $Z$ is a hypersurface given by an explicit equation inside $\pr_S \times_S \mathbb{A}^{0|2} \cong \pr \times \mathbb{A}^{0|n_R/2}$ (Proposition \ref{zisprojective}).

The space $S$ is an affine superspace of dimension $(0|\nr/2-2)$, matching the odd (at this point, expected) dimension of $\mnr$. In Section \ref{moduliofw} we define the  moduli superstack $M(1-\nr/2)$ of families of genus zero supercurves with fiber isomorphic to $\wy$. We then prove in Theorem \ref{formalsmooth} that that the map $S \to M(1-\nr/2)$ associated to $Z$ is a smooth cover. 
Finally, we use the natural action of the group superscheme $\awy$ (described in Section~\ref{automorphismofwy}) on $S$, to give an explicit construction of $M(1-\nr/2)$ as a quotient superstack $[S/\awy]$ in Theorem \ref{artin}. 

We have already parameterized the supercurves; it remains to parameterize the
SUSY structures on them.
We start with parameterizing SUSY structures on the universal deformation $Z$.
To do this, we consider the natural morphism $\mnr \to M(1-\nr/2)$ forgetting the SUSY structure and pull the smooth cover $S \to M(1-\nr/2)$ back to $\mnr$:
    \begin{equation} \label{introduction}
\begin{tikzcd}
 S \times_{M(1-\nr/2)} \mnr \arrow[d] \arrow[r] & S \arrow[d] \\
                        \mnr \arrow[r]                    & M(1-\nr/2).                      
\end{tikzcd}
\end{equation} 
We observe that this pullback is a stack parameterizing the SUSY structures on $Z$.  In Section \ref{susystructure}, we give an explicit description of this pullback using the fact that any SUSY line bundle on $\wy$ deforms to a unique, up to isomorphism, line bundle on $Z$ (Proposition \ref{otwodeformsuniquely}). 
We notice that the chart 
description of the line bundle $\mc{O}(n)$ (transition function $z^{-n}$ over the two standard overlapping charts) over $\wy$ can still be used to define a line bundle over $Z$, which we denote $\mc{O}_Z(n)$.
In particular, it follows 
that a SUSY line bundle over $Z$ must be isomorphic to the deformation  $\mc{O}_Z(-2)$ of the aforementioned $\mc{O}_{\wy}(2)$. 
The moduli space of SUSY structures on 
$Z$, just like that on $\wy$ above, turns out to be the quotient of an open subsuperscheme $Y$ (described in Section \ref{zassrs}) of $\mathbb{H}^0(Z, \Omega_{Z/S} \otimes \mc{O}_Z(2))$, the super vector space $H^0(Z, \Omega_{Z/S} \otimes \mc{O}_Z(2))$ viewed as an affine superspace, by the group superscheme of global invertible functions $\Gamma^*_Z := \mathbb{H}^0(Z, \mc{O}_Z^*)$. In particular, the pullback in \eqref{introduction} is represented by this quotient (Theorem \ref{quotientequal}) .

The action of $\awy$ on $S$ naturally lifts to an action on $Y/\Gamma^*_Z$. Moreover,
the diagram \eqref{introduction} is now identified as
\begin{center}
\begin{tikzcd}
Y/\Gamma^*_Z \arrow[r] \arrow[d] & S \arrow[d]        \\
\mnr \arrow[r]                    & {M(1-\nr/2)=[S/\awy]},
\end{tikzcd}
\end{center}
with the morphism $Y/\Gamma^*_Z \to \mnr$ being 
an $\awy$-torsor; and thus
$\mnr = [(Y/\Gamma^*_Z)/\awy]$ (Theorem \ref{maintheorem}).

The last section of the paper, Section \ref{automorphisms}, proves that the stabilizers of the action of $\awy$ on $Y/\Gamma^*_Z$ are isomorphic to $\Z_2$, and so we finally conclude that $\mnr$ is a Deligne-Mumford superstack (Corollary \ref{finat}).

Thus, the construction of the supermoduli space $\mnr$ is rather straightforward. First, take the universal deformation space $S$ of the supercurve $\wy$. The space $S = \mathbb{H}^1(\wy, \mc{T}_{\wy})$ is the affine superspace associated with the super vector space $H^1(\wy, \mc{T}_{\wy})$. Then take the universal curve $Z$ over $S$ and consider the affine superspace $W = \mathbb{H}^0(Z, \Omega_{Z/S}(2))$ over $S$. Take the open subsuperscheme $Y$ of $W$ defined as the complement to a certain discriminant locus of $W$. Take the quotient $Y/\Gamma^*_Z$ of $Y$ by a natural action of $\Gamma^*_Z = \mathbb{H}^0(Z, \mc{O}_Z^*)$ and then mod out once again by a natural action of $\awy$ to get $\mnr =  [(Y/\Gamma^*_Z)/\awy]$. The hard part is to show that the resulting superstack $\mnr$ is indeed a Deligne-Mumford superstack representing the supermoduli functor. The bulk of the paper is dedicated to demonstrating this.

This paper has had a long life since the first arXiv posting, with which the authors were never satisfied, given that the construction of the supermoduli space was outlined, but only the existence of the corresponding Deligne-Mumford superstack was proven. This is a considerably revamped version of the original paper. Since the first version of the paper, a number of important papers on the moduli superstack of punctured super Riemann surfaces appeared: \cite{bruzzo2021supermoduli} proves that with a suitable level structure, the supermoduli space of punctured super Riemann surface of genus $g \ge 2$ is an algebraic superspace and \cite{moosavian2019existence} shows that the supermoduli space of stable super Riemann surfaces ($2g-2 + \ns + \nr > 0$) is a Deligne-Mumford superstack, among other things.

\subsection*{Acknowledgments}

We are grateful to Christine Berkesch, Ugo Bruzzo, Ionu\c{t} Ciocan-Fontanine, Giulio Codogni, Daniel J. Diroff, Ron Donagi, Rita Fioresi, and Daniel Hern\'andez Ruip\'erez for helpful discussions. N.~O. thanks the Mittag-Leffler Institute and the Centre for Quantum Mathematics as the University of Southern Denmark for hospitality at various stages of work on the project. A.~A.~V. expresses his gratitude to New York University Abu Dhabi for hospitality during his sabbatical in 2019. The work of the first author was supported by the Simons Postdoctoral grant at the University of Pennsylvania.
The work of the second author was supported by World Premier International Research Center Initiative (WPI), MEXT, Japan, and a Collaboration grant from the Simons Foundation (\#585720).

\section{Notation and Conventions} 
 
The letter $k$ will denote an algebraically closed field $k$ of characteristic zero.  All superschemes will be assumed to be over $\on{Spec} k$, unless otherwise specified. All superschemes are also assumed to be Noetherian and locally of finite type. We will work with the following categories: 

\begin{quote}
$\sS=$ the category of superschemes over $ \on{Spec} k$. 
\end{quote}

\begin{quote}
$\on{sArt}_{k}=$ the category of local Artin $k$-superalgebras with residue field $k$.
\end{quote}

\section{Algebraic Supergeometry}

In this section we give a brief introduction to algebraic supergeometry,  following \cite{vaintrob1990deformation, varadarajan2004supersymmetry, carmeli2011mathematical}. 
 
 \begin{definition}(Superspace). A \emph{superspace} is a \emph{locally super ringed space} $(X, \mc{O}_X)$, \emph{i.e}., a pair consisting of a topological space $X$  and a structure sheaf 
 \[ \mc{O}_X = \mc{O}_X^+ \oplus \mc{O}_X^-\]
 of supercommutative rings such that the stalks are local supercommutative rings. The component $\mc{O}_X^+$ (respectively, $\mc{O}_X^-)$ is called the sheaf of \emph{even (resp. odd) (super)functions}.
 \end{definition}

We write $\mc{J}= \mc{O}_X^- \oplus (\mc{O}_X^-)^2$ for the ideal in the structure sheaf $\mc{O}_X$ generated by $\mc{O}_X^{-}$. The ringed space $X_{\bos}:=(X, \mc{O}_X/\mc{J})$, which is a purely even subspace of $X$, is called the \emph{bosonic reduction} of $X$.  
If the closed immersion $i: X_{\bos} \hr X$ induced by $\mc{O}_X \to \mc{O}_X/\mc{J}$ has a retraction $\pi: X \to X_{\bos}$, then we say that  $X$ is \emph{projected}. 

A stronger condition on $X$ than projectedness is that of splitness. Roughly, $X$ is split if the retraction map $\pi: X \to X_{\bos}$ is a vector bundle, whose fibers will automatically be purely odd super vector spaces. More precisely, we say that $X$ is \emph{split} if $\mc{O}_X$ is isomorphic to the sheaf
\[ \on{Gr}_{\mc{J}}(\mc{O}_X)= \mc{O}_X/\mc{J} \oplus \mc{J}/\mc{J}^2 \oplus \mc{J}^2/\mc{J}^3 \oplus \cdots, \] 
of supercommutative rings via an isomorphism inducing the identity on the common quotient $\mc{O}_{X_{\bos}} =\mc{O}_X/\mc{J}$.  
Any study of split superspaces can be reduced to studying vector bundles over $X_{\bos}$. However, superspaces are not, in general, split. For example, Donagi and Witten show in \cite{donagi2015supermoduli} that $\mathfrak{M}_{g, \ns}$ is not split (or even projected) whenever $g \ge 5$.

\begin{definition}[Superscheme] A \emph{superscheme} $(X, \mc{O}_X)$ is a superspace $(X, \mc{O}_X)$ such that $(X, \mc{O}_X^+)$ is a scheme (over $k$) and $\mc{O}_X^-$ is a quasi-coherent sheaf of $\mc{O}_X^+$-modules. A \emph{morphism of superschemes} $f:(X, \mc{O}_X) \to (Y, \mc{O}_Y)$ is a continuous map $f:X \to Y$ of topological spaces along with a morphism $f^{\#}: \mc{O}_Y \to f_* \mc{O}_X$ of sheaves of supercommutative $k$-algebras that is local, \emph{i.e}., induces a local morphism between the stalks.
\end{definition}

\begin{remark} \normalfont The quasi-coherence condition comes from the characterization of complex superspaces (see \cite{vaintrob1990deformation}, Proposition 1.1.3) as superspaces for which $\mc{O}_X^-$ is a coherent sheaf of $\mc{O}_X^+$-modules. Note, that we only require $\mc{O}_X^-$ to be quasi-coherent. We will call a superscheme \emph{Noetherian} if the scheme $(X, \mc{O}_X^+)$ is Noetherian and $\mc{O}_X^-$ is a coherent sheaf of $\mc{O}_X^+$-modules. 
\end{remark}

\begin{example}[Affine Superscheme] For a supercommutative ring $R$, the standard construction of an affine scheme generalizes easily to give an affine superscheme $\on{Spec} R$, see \cite{leites1983theory,manin1988complex,carmeli2011mathematical}.
\end{example}

. These  are defined below. We will also define weighted projective superspaces, as these spaces will play a central role in the moduli problem. 

\begin{example}[Affine Superspace] \normalfont Given a super vector space $V$ of dimension $m|n$ over the ground field $k$, we define the corresponding \emph{affine superspace} as
\[
\mathbb{V} = \on{Spec} \mc{S} (V^\vee),
\]
where $V^\vee$ stands for the dual super vector space and $S$ denotes the supersymmetric algebra.
The standard \emph{affine $(m|n)$-superspace} is the affine superspace $\mathbb{A}^{m|n}$ associated with the vector superspace $V = k^{m|n}$. If we choose coordinates $x_1, \dots, x_m, \theta_1, \dots, \theta_n \in V^\vee$ on $V$, then $V$ may be identified with $k^{m|n}$ and
\[ \mathbb{V} \cong \mathbb{A}^{m|n} = \on{Spec} k[x_1, \cdots, x_m \; | \;  \theta_1, \cdots, \theta_n]. \] 
The bosonic reduction of $\mathbb{A}^{m|n}$ is the affine space $\mathbb{A}^m$. The super affine space is split and may be regarded as a super ``ringification" of its bosonic reduction, the affine space $\mathbb{A}^m$, with a structure sheaf 
\[ \mc{O}_{\mathbb{A}^{m|n}} = \mc{S}_{\mc{O}_{\mathbb{A}^m}} (\Pi \mc{O}_{\mathbb{A}^m}^n),
\]
where $\Pi$ denotes the parity change operation, which shifts the grading by 1 modulo 2. The supersymmetric algebra of a purely odd linear object, in this case the sheaf $\Pi \mc{O}_{\mathbb{A}^m}^n$ of $\mc{O}_{\mathbb{A}^m}$-modules, is well-known as the exterior (Grassmann) algebra of the corresponding purely even object:
\[
\mc{S} (\Pi \mc{O}_{\mathbb{A}^m}^n) = \bigwedge(\mc{O}_{\mathbb{A}^m}^n).
\]
We will avoid using the exterior algebra, because its natural extension to the super world is superanticommutative rather than supercommutative.
\end{example}

\begin{example}[Projective Superspace] \normalfont Given a super vector space $V = V^+ \oplus V^-$ of dimension $m+1|n$, the $(m|n)$-dimensional super projective space $\mathbb{P}(V)$ may be defined as the superspace of lines, \emph{i.e}., $(1|0)$-dimensional vector subspaces of $V$. More technically, $\mathbb{P}(V)$ is the superscheme representing the functor of points that assigns a superscheme $T$ the set of \emph{locally supplemented} $(1|0)$-line subbundles\footnote{Here and henceforth a \emph{super vector bundle} is a locally free sheaf and a \emph{subbundle} is a locally free subsheaf. A \emph{morphism of vector bundles} is a just a morphism of locally free sheaves.} of the trivial super vector bundle $T \times V$ over $T$. These are linear duals $\mc{L}^\vee$ of locally free, rank-$(1|0)$ quotient sheaves $\mc{L}$ of $\mc{O}_T \otimes V^\vee$ on $T$. One may also think of $\mathbb{P}(V)$ as the quotient
\[
\mathbb{P}(V) = (\mathbb{V}\setminus \{0\})/\mathbb{G}_m
\]
of the super affine space $\mathbb{V}$ with deleted origin by the \emph{multiplicative group} $\mathbb{G}_m = \on{GL}(1)$, acting on $\mathbb{V}\setminus \{0\}$ by dilations. Finally, the construction of a projective spectrum generalizes to the super case, and one may identify
\[ \mathbb{P}(V) := \Proj \mc{S}(V^\vee), \]
where the algebra $\mc{S}(V^\vee)$ is $\mathbb{Z} \oplus \mathbb{Z}_2$-graded with the $\mathbb{Z}$-grading coming from the symmetric powers and the $\mathbb{Z}_2$-grading coming from the $\mathbb{Z}_2$-grading on $V^\vee$, the supercommutativity engaging only the $\Z_2$-grading. 
It turns out the super projective space over $k$ is isomorphic to a split superscheme defined by the ordinary projective space $\mathbb{P}(V^+)$ and structure sheaf
\[ \mc{O}_{\mathbb{P}(V)} \cong \mc{S}((V^-)^\vee \otimes \mathcal{O}_{\mathbb{P}(V^+)}(-1)), \]
see \cite[Proposition 4.3.5]{manin1988complex}.
If $V = k^{m+1|n}$ with coordinates $x_0, \dots, x_m \, | \, \theta_1, \dots, \theta_n$, then 
\[ \mathbb{P}^{m|n} = \Proj k[x_0, \cdots, x_m \, | \,  \theta_1, \cdots, \theta_n], \]
where all the generators have degree one in the $\mathbb{Z}$-grading.
\end{example}

\begin{example}[Weighted Projective Superspace] \label{wsp} \normalfont Given a weight vector $(a_0, \dots, a_m \, | \,  \alpha_1, \dots, \alpha_n) \in \mathbb{Z}^{m+1|n}$ with $a_i > 0$, $i = 0, \dots, m$, the \emph{weighted super projective space} $\mathbb{W}\mathbb{P}^{m|n}(a_0, \dots, a_m \, | \,  \alpha_1, \dots, \alpha_n)$ is defined as
\[ \mathbb{W} \mathbb{P}^{m|n}(a_0, \dots, a_m \, |\,  \alpha_1, \dots, \alpha_n) := \Proj k[x_0, \cdots, x_m \, | \,  \theta_1, \cdots, \theta_n], \] 
where the $\mathbb{Z}$-grading on the polynomial superalgebra $k[x_0, \cdots, x_m \, | \, \theta_1, \cdots, \theta_n]$ is defined by declaring the degree of $x_i$ to be $a_i$ and the degree of $\theta_j$ to be $\alpha_j$ for all $i$ and $j$. One may also identify the weighted super projective space as the quotient
\[
(\mathbb{A}^{m+1|n} \setminus \{0\})/\mathbb{G}_m,
\]
where the multiplicative group acts on the super affine space according to the given weights:
\[
g (x_0, \dots, x_m \, |\, \theta_1, \dots, \theta_n) = (g^{a_0} x_0, \dots, g^{a_m} x_m \, | \, g^{\alpha_1} \theta_1, \dots, g^{\alpha_n} \theta_n).
\]
We will be interested in particular in the weighted super projective space
$\wm$ for $m \in \Z$, in which case it is also split and isomorphic to 
the ordinary projective line $\mathbb{P}^1$ equipped with the structure sheaf
\begin{equation}
    \label{wm-functions}
\mc{O}_{\wm} = \mc{S} \left( \Pi \mc{O}_{\mathbb{P}^1}(-m) \right) = \mc{O}_{\mathbb{P}^1} \oplus \Pi \mc{O}_{\mathbb{P}^1}(-m).
\end{equation}
\end{example}

\section{Supercurves} \label{Supercurve}

A \emph{supercurve} $X$ is a $(1|1)$-dimensional, smooth (\emph{i.e}., $\Omega_X^1$ is a locally free sheaf of $\mc{O}_X$-modules of rank $(1|1)$),  proper, connected superscheme. Its bosonic reduction $X_{\bos}$ is an ordinary curve. The \emph{genus} of a supercurve $X$ is equal to the genus of $X_{\bos}$.  We will be mostly interested in families of supercurves over superschemes.

\begin{definition}[Relative Supercurve] A \emph{family of supercurves over a superscheme $T$} is a smooth, proper morphism $\pi: X \to T$ of superschemes of relative dimension $(1|1)$ whose fibers $X_t$ are supercurves. 
\end{definition}
We will often call a family of supercurves over $T$ just a \emph{supercurve over} $T$.

Examples of \emph{genus-zero} supercurves include the \emph{super projective line} $\pp^{1|1} = \wo$ and the weighted projective superspaces $\wm$. The next lemma shows that this set of examples is a complete classification of genus-zero supercurves.

\begin{lemma} \label{scm}  Let $X$ be a genus-zero supercurve over $k$. Then there exists a unique $m \in \mathbb{Z}$ such that $X \cong \wm$.   
  
  \end{lemma} 
  
  \begin{proof} Let $\mc{J}= \mc{O}_X^- \oplus (\mc{O}_X^-)^2$ be the ideal in the structure sheaf $\mc{O}_X = \mc{O}_X^+ \oplus \mc{O}_X^-$ generated by the odd elements. Since $X$ is $(1|1)$-dimensional, $\mc{J}^2=(\mc{O}_X^-)^2=0$ and so $\mc{J}=\mc{O}_X^-$. 
  The associated graded ring  $\on{Gr}_{\mc{J}}(\mc{O}_X)=\mc{O}_X/\mc{J} \oplus \mc{J}$ is, therefore, equal to $\mc{O}_{X_{\bos}} \oplus \mc{O}_X^-$, implying that $X$ is split, \emph{i.e}., $X$ is the total space of the odd line bundle $\mc{J}^\vee$ over $X_{\bos} \cong \pp^1$.  There exists $m \in \Z$ such that $\mc{J}^\vee \cong \Pi \mc{O}_{\pp^1}(m)$ and, therefore, $\mc{O}_X \cong \mc{S}\left(\Pi \mc{O}_{\pp^1}(-m)\right)$. 
  The isomorphism $X \cong \wm$ then follows from Example \ref{wsp}.
  
  \end{proof} 
  
The above lemma shows that the reduced space underlying the $\wm$ is independent of $m$ and isomorphic to $\pp^1$. The curve $\pp^1$ is well-known to be rigid--- that is, any first order deformations of $\pp^1$ is isomorphic to the trivial deformation of $\pp^1$, or equivalently, $h^1(\pp^1, \mc{T}{\pp^1})=0$. 
This rigidity does not extend to $\wm$: 
 \begin{lemma} \label{dimensions}
 For $X = \wm$, we have
 \begin{align*}
  \on{dim}_k H^1(X, \mc{T}X) & = \begin{cases} 0 \,| -m-1 & \  \text{for} \ m <-1, \\ 
0\,|\,0 & \  \text{for} \  -1 \le m \le 3, \\
 0\,|\,m-3 & \  \text{for} \  m > 3. 
\end{cases}
      \end{align*}
In particular, for $ m < -1$ and $m > 3$, the weighted projective superspaces $\wm$ are not rigid.
     \end{lemma} 
  
This is a straightforward \v{C}ech cohomology
computation.

\begin{remark} Skipping ahead, let $n = \nr$  be the number of Ramond punctures on a genus zero SUSY curve $\Sigma$. We will soon see that the weighted projective superspace $\wb$ is the supercurve underlying $\Sigma$. Note that the dimension $(0\,|\,\nr/2-2)$ of the space of first-order deformations matches the expected odd dimension of $\mnr$. 
\end{remark} 

\section{Weighted Projective Superspace} \label{weightsuperproj}

In this section, let $n = \nr \ge 4$ be an even number, which is meant to be the number of Ramond punctures in the sequel. The superscheme $\wy:= \wn = \Proj A$ for the $\Z$-graded superalgebra $A := k[u,v\,|\, \theta]$ with $u, v$ having degree 1 and $\theta$ having degree $1-n/2$ and supercommutativity taking into account only the $\Z_2$-grading, admits an affine covering $\{U(u \neq 0), V(v \neq 0)\}$, with
 \begin{align*}
 \phi_u:  \wy \vert_U & \lgr \on{Spec} k[z\, | \,  \zeta], \\
\nonumber v/u & \mapsto z, \\
\nonumber \theta u^{n/2-1} & \mapsto \zeta,
 \end{align*}
 and\
 \begin{align*}
 \phi_v: \wy \vert_V & \lgr  \on{Spec} k[w,\chi], \\
\nonumber  u/v & \mapsto w, \\
\nonumber \theta v^{n/2-1} & \mapsto \chi.
 \end{align*}
These charts are glued together over $U \cap V$ using the automorphism
\begin{align} \label{gluingfunctionforwp}
\phi_v \circ \phi_u^{-1}:\on{Spec} k[z \, | \,  \zeta] & \lgr \on{Spec} k[w \, | \, \chi], \\
\nonumber z  & \mapsto 1/w, \\ 
\nonumber \zeta & \mapsto  \chi w^{n/2-1}.
\end{align}

 v

\subsubsection*{Picard Group}

 For each $m \in \mathbb{Z}$, we define the line bundle $\mc{O}(m)$ on $\wy=\Proj A$ to be the invertible sheaf $A(m)^{\sim}$ associated with the degree-shifted graded $A$-module $A(m)$. 
To give an example, the line bundle $\mc{O}(1)$ on $\wy$ is generated by its global sections, $\Gamma(\wy, \mc{O}(1)) = A(1)^0 = A^1$, the elements of $A(1)$ of degree 0. One can choose the elements $\{ u,v, u^{n/2-i} v^i \theta \; | \; 0 \le i \le n/2 \} $ as a basis of $\Gamma(\wy, \mc{O}(1))$. Let $\on{Pic}(\wy)$ be the \emph{Picard group}, the group of isomorphism classes of rank-$(1|0)$, invertible sheaves on $\wy$.

\begin{lemma}
  \label{Picard}
$\on{Pic}(\wy) \cong \Z$. 
\end{lemma}

\begin{proof}
Just like in classical geometry, the Picard group of a superscheme $X$
may be computed as the first cohomology group $H^1(X, (\mc{O}_X^+)^*)$
of the sheaf $(\mc{O}_X^+)^*$ of invertible even functions on $X$. As
always, this is easier to do if we engage Serre's GAGA principle and
use the sheaf of holomorphic functions on $\wy$ with complex topology,
as in \cite{witten2012notes}, because it allows us to consider the
exponential sequence $0 \to \Z \to \mc{O}_{\wy}^+ \to
(\mc{O}_{\wy}^+)^* \to 0$, to which we apply $H^0(\wy, -)$ and get the
long exact sequence
\begin{multline*}
 0 \longrightarrow H^0(\Z) \longrightarrow H^0(\mc{O}_{\wy}^+) \longrightarrow H^0((\mc{O}_{\wy}^+)^*) \longrightarrow H^1(\Z) \longrightarrow H^1(\mc{O}_{\wy}^+) \\
 \longrightarrow H^1((\mc{O}_{\wy}^+)^*) \overset{\on{deg}}{\longrightarrow} H^2(\Z) \cong \Z \longrightarrow 0.
 \end{multline*}
From \eqref{wm-functions}, we find  that $H^1(\mc{O}_{\wy}^+)=0$ and so $H^1((\mc{O}_{\wy}^+)^*) \cong \mathbb{Z}$.
\end{proof}

Let $\mc{L}$ be a line bundle on $\wy$ and $i^* \mc{L}$ denote its pullback by the inclusion $i: \pr \hr \wy$. 

\begin{corollary}
  \label{susybundleisontwo}
If $i^* \mc{L} \cong \mc{O}_{\pr}(m)$ for some $m \in \Z$, then
$\mc{L} \cong \mc{O}_{\wy}(m)$.
\end{corollary}

\begin{proof}
Suppose $ i^* \mc{L} \cong \mc{O}_{\pr}(m)$. Since $i^*: \on{Pic} (\wy) \to \on{Pic} (\pr)$  is an isomorphism and
\[ i^*\mc{O}_{\wy}(m) = \mc{O}_{\wy}(m)_{\bos} = (A(m)_{\bos})^{\sim}  = A_{\bos} (m)^{\sim} = \mc{O}_{\pr}(m), \] 
we have
$\mc{L} \cong \mc{O}_{\wy}(m)$. 
\end{proof}

\begin{proposition} \label{otwodeformsuniquely}
The line bundle $\mc{O}_{\wy}(m)$ for each $m \in \Z$ deforms uniquely, up to isomorphism.
\end{proposition} 

\begin{proof} 

The deformation theory of $\mc{O}_{\wy}(m)$ is controlled by the cohomology of its endomorphism sheaf, $\mc{O}_{\wy}$. Since
$H^2(\mc{O}_{\wy}) \linebreak[1] = \linebreak[0] H^1(\mc{O}_{\wy}) = 0$ it follows that there exists a unique, up to isomorphism, deformation of $\mc{O}_{\wy}(m)$.  
\end{proof}

\subsubsection*{Automorphisms} \label{automorphismofwy}

Recall that $\wy = \wn$ is the Proj superscheme $\Proj A$ of the graded superalgebra $A:=k[u,v \, | \,  \theta]$ with $\mathbb{Z}$-grading $|u|=|v|=1$ and $|\theta| \linebreak[2] = \linebreak[1] 1 \linebreak[0] - n/2$. If we forget the grading on $A$ and take corresponding affine superscheme $\Spec A$, we obtain the affine superspace $\mathbb{A}^{2|1}$.
\bigskip

\noindent \emph{General Strategy}. To describe the group superscheme $\Aut \wy$, we will follow the strategy of \cite{amrani1989classes}, where the automorphism groups of classical weighted projective spaces are computed.
Since $\wy$ is the GIT quotient $\mathbb{A}^{2|1}/\!\!/ \mathbb{G}_m$, where $\mathbb{G}_m$ acts on $u$ and $v$ with weight 1 and on $\theta$ with weight $1-n/2$, we may expect that 
\[ \awy = \Aut(A)/\Gamma^*, \] 
where $\Aut(A)$ is the supergroup of automorphisms of the homogeneous algebra $A$ and $\Gamma^*$ is some supergroup making up for the action of $\mathbb{G}_m$. 

We will first identify $\Gamma^*$ as a group superscheme and then find a short exact sequence of group superschemes:
\begin{equation} \label{hope} 1 \longrightarrow \Gamma^* \longrightarrow \Aut (A) \longrightarrow \awy \longrightarrow  0.  \end{equation}

The existence of such a short exact sequence would imply that $\Gamma^*$ is a normal subgroup of $\Aut(A)$ and $\awy = \on{Aut}(A)/\Gamma^*$.

We will now give functorial descriptions of the group superschemes appearing in  \eqref{hope}. 

\begin{definition}
\label{gamma}
Let $\Gamma^*$ denote the functor
\begin{align}
    \Gamma^*: \sS & \to \on{Group}, \\
   \nonumber T & \mapsto \Gamma((\mc{O}_{\wy \times T}^+)^*),
\end{align}
\end{definition} 

\begin{definition} Let  $\Aut (A)$ denote the functor
\begin{align*}
\Aut (A): \sS & \to \on{Group}, \\
\nonumber T & \mapsto \Aut_R (A \otimes_{k} R),
\end{align*}
where $R= \Gamma(T, \mc{O}_T)$. 
\end{definition}

\begin{definition} Let  $\awy$ denote the functor
\begin{align*} 
\awy: \sS & \to \on{Group} \\
\nonumber T & \mapsto \Aut_T (\wy \times T)
\end{align*} 
where $\Aut_T (\wy \times T)$ is the group of automorphisms of $\wy \times T$ over $T$. 
\end{definition}
Our next goal is to show that these functors are indeed representable.

\begin{lemma} The functor $\Aut (A)$ is represented by an affine group superscheme based on the superscheme
\begin{equation}
\label{AutA}
\on{Spec}k[a,b,c,d,e, e^{-1}, (ad-bc)^{-1} \, | \, \alpha_0, \cdots, \alpha_{n/2}, \beta_0, \cdots, \beta_{n/2}]
\end{equation}
of dimension $(5\, |\, n+2)$. 
\end{lemma}

\begin{proof}  The set of $T$-points of $\Aut(A)$ consists of the automorphisms
\begin{align}
\label{aut11}
u & \mapsto  au + bv + \theta  \sum_{i=0}^{n/2} \alpha_i u^{n/2-i}v^{i},  &  a, b \in R^+, \alpha_i \in R^-, \\
\nonumber v & \mapsto  c u + d v + \theta \sum_{j=0}^{n/2} \beta_j u^{n/2-j}v^j, & a',b' \in R^+, \beta_j \in R^-,\\
\nonumber  \theta & \mapsto  e \theta, & e  \in (R^+)^*,
\end{align}
where $R = \Gamma(T, \mc{O}_T)$ and $ad -bc \neq 0$.   It is immediate from the above description that, as a set-valued functor, $\Aut (A)$ is represented by the distinguished open subset $D(ad-bc)$ of the affine superscheme 
\begin{equation*}
 \on{Spec}k[a,b,c,d,e, e^{-1} \, | \, \alpha_0, \cdots, \alpha_{n/2}, \beta_0, \cdots, \beta_{n/2}].
\end{equation*}
\end{proof}

\begin{remark}
The Hopf superalgebra structure on the algebra of regular functions on the affine superscheme \eqref{AutA} can be described by composition of the automorphisms given in \eqref{aut11}. 
\end{remark}

The group superscheme $\Gamma^*$ is easier to describe explicitly. We will do that in terms of the standard multiplicative group $\mathbb{G}_m$ and the \emph{odd additive group} $\mathbb{G}_a^{0|1}$, which is a group superscheme representing the functor 
\begin{align*}
  \mathbb{G}_a^{0|1}:  \sS & \to \on{Group},\\
T &\mapsto \Gamma (T, \mc{O}_{T}^-),
\end{align*}
with the group law given by addition. The underlying superscheme of $\mathbb{G}_a^{0|1}$ is just $\mathbb{A}_k^{0 \, | \, 1}$.

\begin{lemma}
\label{Gamma*}
The functor $\Gamma^*$ is represented by the group superscheme $\mathbb{G}_m \times (\mathbb{G}_a^{0|1})^{n/2}$ of dimension $(1|n/2)$. 
\end{lemma} 

\begin{proof} Given an affine superscheme $T= \on{Spec} R$, we may compute the group $H^0((\mc{O}_{\wy \times T}^+)^*)$ explicitly using \v{C}ech cohomology, and describe its elements in local coordinates $(z \, | \, \zeta)$ on $U$ as 
\begin{equation} \label{element} a_0 \left(1 + \zeta \sum_{i=0}^{
n/2-1} \beta_i z^i \right)  \end{equation} 
with $a_0 \in (R^+)^*$ and $\beta_i \in R^{-}$. 
The multiplication on the $T$-points of $\Gamma^*$ is given by
\begin{equation}\label{mult} a_0 \left(1+ \theta \sum_{i=0}^{n/2-1} \beta_i z^i \right)\cdot a_0' \left(1+ \zeta \sum_{i=0}^{n/2-1} \beta_i' z^i \right) = a_0 a_0'\left( 1 +  \zeta \sum_{i=0}^{n/2-1} (\beta_i + \beta_i') z^i \right) \end{equation}  
with $a_0, a_0' \in (R^+)^*$ and $\beta_i, \beta_i' \in R^{-}$. The lemma follows from Equations \eqref{element} and \eqref{mult}.
\end{proof}

\begin{lemma} The group superscheme $\Gamma^*$ is a  group subsuperscheme of $\Aut (A)$.

\end{lemma} 

\begin{proof} We can describe the each element of the group $\Gamma^*(T)$ explicitly by  \begin{equation} r_0 \left(1 + \theta \sum_{i=0}^{n/2-1} \beta_i v^i u^{n/2-1-i} \right) \end{equation} with $r_0 \in (R^+)^*$ and $\beta_i \in R^{-}$  
and identify each such element with the automorphism in $\Aut (A \otimes R)$ sending 
\begin{align}
\label{A}
 u & \mapsto r_0(1 + \theta \sum_{i=0}^{n/2-1} \beta_i  v^i u^{n/2-i-1})u ,\\
\nonumber v  & \mapsto r_0(1 + \theta \sum_{i=0}^{n/2-1} \beta_i  v^i u^{n/2-i-1}) v,\\
\nonumber \label{C}\theta & \mapsto r_0^{1-n/2} (1 + \theta \sum_{i=0}^{n/2-1} \beta_i  v^i u^{n/2-i-1})^{1-n/2} \theta\\
\nonumber
& \quad = r_0^{1-n/2} \theta. \qedhere
\end{align}
\end{proof} 

\begin{theorem} \label{rep}  The functor $\awy$ is represented by the group superscheme $\on{Aut}(A)/\Gamma^*$ of dimension $(4 \, | \, n/2+2)$. 
\end{theorem}

\begin{proof} 

When mentioning elements of group superschemes and points of superschemes, we will be talking about $T$-points thereof without specifying the ``probe'' superscheme $T$.

Every automorphism of the superalgebra $A=k[u,v\, | \, \theta] $ induces an automorphism of $\wy$ so that we have a group superscheme map $ \Aut(A) \to \awy$.  
To prove the theorem, it suffices to show that the kernel of the above map is isomorphic to $\Gamma^*$ and that the map is surjective.

Lets us start with identifying the kernel. Referring to the description \eqref{aut11} of the action of $\Aut (A)$, we observe that a point of $\wy$ with homogeneous coordinates $[u:v \, | \, \theta]$ maps to itself by an automorphism of $A$, if and only if it is given by
\begin{align*}
u & \mapsto  ( a + \theta  \sum_{i=0}^{n/2-1} \alpha_i u^{n/2-i-1}v^{i} ) u, \\
v & \mapsto  ( a + \theta  \sum_{i=0}^{n/2-1} \alpha_i u^{n/2-i-1}v^{i}) v,\\
 \theta & \mapsto  ( a + \theta  \sum_{i=0}^{n/2-1} \alpha_i u^{n/2-i-1}v^{i})^{1-n/2} \theta.
\end{align*}
This easily identifies with \eqref{A}, which describes the inclusion of $\Gamma^*$ in $\Aut (A)$.

Since group scheme quotients and surjective morphisms are usually well-defined in the fppf topology, we will use the same formalism in the category of superschemes. Thus, to show the surjectivity of the group superscheme morphism $\Aut(A) \to \awy$, we need to show that the morphism of sheaves  in the fppf topology on $\sS$ defined by the corresponding functors of points $T \mapsto \Aut(A)(T)$ and $T \mapsto \awy(T)$ is surjective. Since the fppf topology is finer than the Zariski topology on $\sS$, it will be enough to show the surjectivity in the Zariski topology. This means that for every $k$-superscheme $T$ and every automorphism $\alpha$ of $\wy \times T$ over $T$, we will need to find a Zariski open covering $\mc{U} =\{ U_i \}$ of $T$ such that for all $i$ the restriction $\alpha|_{U_i}$ is in the image of $\Aut(A)(U_i) \to \awy(U_i)$.

Denote $\wy \times T$ by $\wy_T$ and let $\on{pr}_1: \wy \times T \to \wy$ and $\on{pr}_2: \wy \times T \to T$ be the canonical projections. Consider the diagram 

\begin{equation}
\label{diagram1}
\begin{tikzcd}
\wy_T \arrow[rd, "\on{pr}_2"] & \wy_T \arrow[l,"\alpha"'] \arrow[d, "\on{pr}_2"] \arrow[r, "\on{pr}_1"] & \wy  \arrow[d, "\pi"] \\
                       & T  \arrow[r, "f"] &   \on{Spec} k  .
\end{tikzcd} 
\end{equation}

Every line bundle on $\wy_T$ is isomorphic to $\prn^* \mc{O}_{\wy}(l) \otimes \prt^*\mc{L}$ for some $l \in \Z$ and a line bundle $\mc{L}$ on the base $T$ (Lemma \ref{Picard}). 
Therefore, there exists an isomorphism
\begin{equation}\label{isom} \sigma: \alpha^* \prn^* \mc{O}(1)  \lgr \prn^* \mc{O}(l) \otimes \prt^*\mc{L}. \end{equation}
By comparing the ranks of the pushforwards of the above line bundles to $T$, we find that $l=1$. 

Let $\mc{U} =\{ U_i \}$ denote a Zariski open covering of $T$ on which the line bundle $\mc{L}$ in \eqref{isom} trivializes. Thus, over each $U_i$, we obtain an isomorphism
\begin{equation*}
\sigma_i: \alpha_i^* \prn^* \mc{O}(1)  \lgr \prn^* \mc{O}(1), \end{equation*}
where $\alpha_i := \alpha|_{U_i}$.
From that we also get $\mc{O}_{\wy \times U_i}$-module isomorphisms
\[ \sigma_i^{\otimes m }: \alpha_i^* \prn^* \mc{O}(m) \lgr \prn^* \mc{O}(m), \qquad m \in \Z, \]
which sum up to an isomorphism
\[
\bigoplus_m \sigma_i^{\otimes m }: \bigoplus_m \alpha_i^* \prn^* \mc{O}(m) \lgr \bigoplus_m \prn^* \mc{O}(m)
\]
of sheaves of $\Z$-graded $\mc{O}_{\wy \times U_i}$-superalgebras. By adjunction, we have an isomorphism
\[
\bigoplus_m  \prn^* \mc{O}(m) \lgr (\alpha_i)_*\left(\bigoplus_m \prn^* \mc{O}(m)\right).
\]
Since the triangle in \eqref{diagram1} with $T$ replaced by $U_i$ commutes, when we apply cohomology, $(\alpha_i)_*$ disappears, and we get a $k$-algebra isomorphism
\begin{equation}
    \label{iso}
\bigoplus_m  H^0(\prn^* \mc{O}(m)) \lgr \bigoplus_m H^0(\prn^* \mc{O}(m)) .
\end{equation}

Note that the $\Z$-graded superalgebra $A=k[u,v\, | \, \theta] $ is equal to the $\mathbb{Z}$-graded superalgebra
\[ \bigoplus_{m \in \Z} H^0(\mc{O}_{\wy}(m))   H^0(\mc{O}_{\wy}(1-n/2 + j))
, \]
the grading given by $m$. More generally, we have a canonical isomorphism of graded $R_i$-su\-per\-al\-ge\-bras: 
\[ A \otimes R_i = \bigoplus_{m \in \Z} H^0(\on{pr}_1^*\mc{O}_{\wy}(m)),  \]
where $R_i := \Gamma(U_i, \mc{O}_{U_i})$. Then the isomorphism \eqref{iso} gives an $R_i$-isomorphism
\begin{equation}
\label{auto}
A \otimes R_i \lgr A \otimes R_i ,
\end{equation}
which may be viewed as an element of $\Aut (A) (R_i)$. This automorphism induces $\alpha_i$, because $\alpha_i$ is determined by the images of the homogeneous coordinates $u,v \in H^0(\prn^* \mc{O}_{\wy}(1))$ and $\theta \in H^0(\prn^* \mc{O}_{\wy}(1-n/2))$ and these images are determined by \eqref{iso}, equivalent to \eqref{auto}. This concludes the proof of surjectivity.
\end{proof}

\section{Super Riemann Surfaces} 

A \emph{super Riemann surface} (SUSY curve) $\Sigma$ is a supercurve $X$ equipped with a \emph{supersymmetry operator}. Physically, a supersymmetry operator is defined as a square root of the Hamiltonian operator generating time translation. 
Mathematically, a supersymmetry operator is formalized as a local generating section of an \emph{odd distribution}, \emph{i.e}., a subbundle $\mc{D}$ of the tangent bundle $\mc{T}X$. Such a section must be an \emph{odd, maximally nonintegrable} vector field $D$ on $X$ such that $D^2 = \frac{1}{2} [D, D]= \delpz$, where $z$ is an even local coordinate on $X$ and where the brackets denote the supercommutator. Note that we can recover the physical definition of a supersymmetry operator by interpreting the vector field $\delpz$ as generating time translation. 

To get used to the definition of a super Riemann surface, it is helpful to consider the $(1|1)$-dimensional complex superspace $\C^{1|1}$ with coordinates $(z \, | \, \zeta)$. An example of an odd, maximally nonintegrable vector field  on $\C^{1|1}$ is given by  \[ D:=\delpzeta + \zeta \delpz \] 

We summarize this discussion with the following coordinate-free definition of a family of super Riemann surfaces, 

 \begin{definition}[Super Riemann Surface] \label{SRS} A \emph{family $\Sigma = (X/T, \mc{D})$  of genus $g$ super Riemann surfaces over a superscheme $T$} is the data of
  
  \begin{enumerate}[$(1)$]
      \item a smooth, proper morphism $\pi: X \to T$ of superschemes of relative dimension $(1|1)$ with genus $g$ fibers $X_{t}$ and
      
      \item a rank-$(0|1)$ subbundle $\mc{D} \subset \mc{T}_ {X/T}$
  \end{enumerate}
satisfying the condition:
\smallskip

\noindent
      The supercommutator $[ \ , \ ]$ of vector fields, composed with the natural projection defines an isomorphism
      \begin{equation}
      \label{commutator}
  \mc{D}^{\otimes 2} \overset{[ \ , \ ]}{\longrightarrow} \mc{T}_{X/T} \to \mc{T}_{X/T}/\mc{D}.        
      \end{equation}
The subbundle $\mc{D}$ is called an \emph{odd, maximally nonintegrable distribution} or a \emph{SUSY structure}.

 \end{definition}
 
 \begin{definition}[SUSY-Isomorphism] A SUSY-isomorphism $\Sigma/T \lgr \Sigma'/T$ is  a $T$-isomorphism $\phi: X \lgr X'$ of supercurves such that $\phi^* \mc{D}'=\mc{D}$. 
\end{definition}
 
 \begin{remark} We write $\Sigma$ for a super Riemann surface over $\on{Spec} k$ and omit the word family when discussing a family of Riemann surface $\Sigma/T$. 
 \end{remark}

Note that the composition isomorphism \eqref{commutator} is automatically $\mc{O}_X$-linear, unlike the supercommutator map itself.
Thus, we get a short exact sequence of vector bundles:
  \begin{equation} \label{definingsequence}
      0 \lr \mc{D} \hookrightarrow \mc{T}_{X/T} \lr \mc{D}^{\otimes 2} \lr 0. 
  \end{equation}

One can easily show that the above example of an odd, maximally nonintegrable vector field  $D$ on $\C^{1|1}$ is, in a way, canonical: 

\begin{quote}
For every super Riemann surface $\Sigma$ over $k$, there exist local coordinates $(z \, | \, \zeta)$ on $X$, called \emph{superconformal coordinates}, such that the odd distribution $\mc{D}$ in these coordinates is generated by the vector field
\[
D_{\zeta} = \delpzeta + \zeta \delpz.
\] 
\end{quote}

This statement can be generalized to families $\Sigma/T$ using \emph{generalized coordinates} locally in the \'etale topology on $T$, cf.\ \cite[Section 3.3]{bruzzo2021supermoduli}.

  In the absence of odd functions on $T$, a super Riemann surface $\Sigma/T$ is
  equivalent to a \emph{spin curve}, which is a curve equipped with a spin structure.

    \begin{definition}[Spin Curve] \label{spin} A genus $g$ family of spin curves $(\mc{C}/T, \mc{L}, j)$ over an ordinary scheme $T$ is the data of
  
  \begin{enumerate}[$(1)$]
      \item a smooth, proper morphism $\pi: \mc{C} \to T$ of schemes of relative dimension $1$ with genus $g$ fibers $\mc{C}_{t}$, 
      
      \item a line bundle $\mc{L}$, and
      
      \item an isomorphism $j: \mc{L}^{\otimes 2} \lgr \Omega_{\mc{C}/T}^1$. 
  \end{enumerate}
   An isomorphism of spin curves $(\mc{C}/T, \mc{L},j) \lgr (\mc{C}'/T, \mc{L}',j')$ is a pair $(f, \phi)$ with $f: C \lgr C'$ an isomorphism of schemes over $T$ and $\phi: f^*(\mc{L}') \lgr \mc{L}$ an isomorphism of line bundles on $\mc{C}$ compatible with $j$ and $f^*(j')$.  \end{definition}

\begin{proposition}
\label{hodge}
A super Riemann surface $\Sigma = (X/T, \mc{D})$ over an ordinary scheme $T$ determines a spin curve. Conversely, every spin curve gives rise to a super Riemann surface.
\end{proposition}

\begin{proof} Dualizing the short exact sequence \eqref{definingsequence} of super vector bundles and then tensoring it with the $\mc{O}_X$-module $\mc{O}_X/\mc{J}$ produces the short exact sequence \begin{equation} \label{pot} 0 \longrightarrow (\mc{D}_{\bos}^{\vee})^{\otimes 2} \longrightarrow \Omega_{X_{\bos}/T}^1 \oplus \mc{J} \longrightarrow \mc{D}_{\bos}^{\vee} \longrightarrow 0  \end{equation} 
of $\mc{O}_{X_{\bos}}$-modules.

Comparing parities, we find that $ \mc{J} \cong \mc{D}_{\bos}^{\vee}$ and $j: (\mc{D}_{\bos}^{\vee})^{\otimes 2} \lgr \Omega_{X_{\bos}/T}^1$, where $j$ is the canonical isomorphism induced by the supercommutator, see \eqref{commutator}.
The triple $(X_{\bos}/T, \Pi \mc{D}_{\bos}^{\vee}, j)$ is a spin curve over $T$. 

To prove the converse, given a spin curve $(\mc{C}/T, \mc{L}, j)$, take the split supercurve $X = (\mc{C}, \mc{S} (\Pi \mc{L}))$ over $T$. The $(0|1)$-rank super vector bundle $\Pi \mc{L}^\vee$ over $X_{\bos} = \mc{C}$ becomes a subbundle of the rank-$(1|1)$ bundle $\mc{T}_{X/T}^-$ over $X_{\bos}$ in two ways: $\Pi \mc{L}^\vee = \mc{J}^\vee = (\mc{J}/\mc{J}^2)^\vee = (\mc{T}_{X/T})^-_{\bos}$ and $\Pi \mc{L}^\vee \xrightarrow{\Pi \mc{L} \otimes (j^{-1})^\vee} \Pi \mc{L} \otimes \mc{T}_{\mc{C}/T} = \mc{J}(\mc{T}_{X/T})^+_{\bos}$. The sum of these two maps defines an inclusion
\[
\Pi \mc{L}^\vee \hookrightarrow (\mc{T}_{X/T})^-_{\bos} \oplus \mc{J} (\mc{T}_{X/T})^+_{\bos}  = \mc{T}_{X/T}^- \subset \mc{T}_{X/T},
\]
of an $\mc{O}_{X_{\bos}}$-module into an $\mc{O}_{X}$-module. This gives an inclusion of $\mc{O}_X$-modules:
\[
\mc{O}_X \otimes_{\mc{O}_{X_{\bos}}} \Pi \mc{L}^\vee \hookrightarrow \mc{T}_{X/T},
\]
whose image defines a SUSY structure $\mc{D}$ on $X/T$.

\end{proof}

\subsubsection*{Ramond and Neveu-Schwarz Punctures}

On a super Riemann surface we consider two distinct types of punctures. One type, called \emph{Neveu-Schwarz (NS) punctures}, is entirely analogous to the marked points one sees in the ordinary moduli theory of curves. The other type of punctures, called \emph{Ramond punctures}, corresponds to divisors along which the SUSY structure degenerates. Both types of punctures arise as nodes when one considers the compactification of the moduli space of SUSY curves and the moduli space $\spin$ of spin curves. The Ramond nodes are the ramification points of the map $\ov{\mc{S} \mc{M}}_g \to \ov{\mc{M}}_g$
that forgets the spin structure. Both types of punctures play a role in superstring theory in relation to string vertex operators.

\begin{definition}[Family of SRSs]
\label{srswithpunctures} A \emph{family of genus $g$ super Riemann surfaces with $\nr$ Ramond and $\ns$ Neveu-Schwarz (NS) punctures} $(X/T, \{s_i\}_{i=1}^{\ns}, \mc{R}, \mc{D})$ is the data of
\begin{enumerate}[$(1)$]
    \item a smooth, proper morphism $\pi: X \to T$ of superschemes of relative dimension $(1|1)$ with genus $g$ geometric fibers $X_t$, 
    
    \item a sequence of disjoint sections $s_i: T \to X $, $i = 1, \dots, \ns$,
    
    \item an unramified relative effective Cartier divisor $\mc{R}$ of degree $\nr$, called the \emph{Ramond divisor} and whose components are labeled and called the \emph{Ramond punctures}, and
    
    \item a rank-$(0|1)$ subbundle $\mc{D} \subset \mc{T}_{X/T}$ such that the composition of the supercommutator morphism with the natural inclusion and projection
    $$
    \mc{D}^{\otimes 2} \hookrightarrow \mc{T}_{X/T}^{\otimes 2}\overset{[ \ , \ ] }{\lr} \mc{T}_{X/T} \to  \mc{T}_{X/T}/\mc{D}
    $$
    defines an $\mc{O}_X$-module isomorphism
    $$
     \mc{D}^{\otimes 2} \to \left(\mc{T}_{X/T}/\mc{D} \right) (-\mc{R})
    $$
    of $\mc{D}^{\otimes 2}$ with the subsheaf $ \left(\mc{T}_{X/T}/\mc{D} \right) (-\mc{R})$ of  $\mc{T}_{X/T}/\mc{D}$.
\end{enumerate}
An \emph{isomorphism} of families of punctured super Riemann surfaces over $T$ is an isomorphism of underlying supercurves respecting the data $(\{s_i\}_{i=1}^{\ns}, \mc{R}, \mc{D})$.
\end{definition}

Therefore, in the presence of Ramond punctures, the short exact sequence in \eqref{definingsequence} is replaced with
\begin{equation} \label{ramondefiningsequence}
0 \lr \mc{D} \lr \mc{T}_{X/T} \lr \mc{D}^{\otimes 2}(\mc{R}) \lr 0. 
\end{equation}

The supermoduli problem is best approached from an equivalent ``dualized perspective'' on the definition of a family of super Riemann surfaces, as follows.

\begin{definition}[Family of SRSs: Dual Definition]
\label{SRSdef}

A \emph{family of genus $g$, $\nr+\ns$-punctured super Riemann surfaces over $T$} is the data $\Sigma = (X/T, \{s_i\}_{i=1}^{\ns}, \mc{R}, \mc{G})$  of

\begin{enumerate} [$(1)$]
\item  a smooth, proper morphism of $\pi: X \to T$ of superschemes of relative dimension $(1|1)$ with genus $g$ fibers $X_{t}$,

 \item $\ns$ distinct sections $s_i: T \to X $, 

\item a relative effective Cartier divisor $\mc{R}$ of degree $\nr$, unramified over $T$,  with labeled components, and

\item a locally free, rank-$(0|1)$ quotient $\Omega_{X/T}^1 \twoheadrightarrow \mc{G}$ such that $\mc{L}:= \on{ker}(\Omega_{X/T}^1 \twoheadrightarrow \mc{G})$ is a rank-$(1|0)$ subbundle $\mc{L} \subset \Omega_{X/T}^1$ (automatic by the projective dimension argument or the condition \eqref{max-int} below) over $X$ such that 

the composition
\[
\mc{L} \hookrightarrow \Omega_{X/T}^1 \overset{d}{\lr} \Omega_{X/T}^2 \lr (\mc{G})^{\otimes 2},
\]
where $d$ is the relative de Rham differential, establishes an isomorphism
\begin{equation}
\label{max-int}
\mc{L} \xrightarrow{\sim} (\mc{G})^{\otimes 2}(-\mc{R})
\end{equation}
onto the submodule
\[
(\mc{G})^{\otimes 2}(-\mc{R}) \subset (\mc{G})^{\otimes 2}.
\]
\end{enumerate}
As common in supergeometry, we will often drop the word "family" and talk about a \emph{super Riemann surface over $T$}. We will also refer to the structure of a family of super Riemann surfaces with punctures on a given family $X/T$ of supercurves as a \emph{superconformal} or \emph{SUSY structure on $X/T$}.

An \emph{isomorphism} of families of super Riemann surfaces over $T$ is an isomorphism of underlying supercurves respecting the SUSY structures.
\end{definition}

An example of a vector field $D_{\zeta}$ generating the subbundle $\mc{D}$ near a Ramond puncture is $D_{\zeta} = \delpzeta + z \zeta \delpz$, where $z$ will define the Ramond divisor $\mc{R}$ locally.
In fact, this example is canonical in the following sense, see \cite[Proposition 3.6]{bruzzo2021supermoduli}: 

\begin{quote}
    For every $\nr$-punctured super Riemann surface $\Sigma$ over a base $T$, \'{e}tale locally on the base $T$ and Zariski locally on $X$, near a closed point of the Ramond divisor $\mc{R}$, there exist local coordinates $(z\, | \, \zeta)$ on $X$ such that $\mc{D}$ in these coordinates is generated by the odd vector field
\[
D_{\zeta} = \delpzeta + z \zeta \delpz
\]
and $\mc{R}$ is defined by the function $z$.
\end{quote}

Punctures on  $X/T$ are closed subsuperschemes of relative codimension $(1|1)$, while a divisors is defined to have codimension $(1|0)$. If we take  $X$ to be a super Riemann surface with \textbf{no} Ramond punctures, then there exists the following duality between prime divisors and punctures on $X$:
For each puncture $p \in X$, issue an odd curve in the direction of $\mc{D}_p$ to produce a closed subscheme of relative dimension $(0|1)$ passing through $p$. Conversely, an odd curve will contain a unique point $p$ at which the tangent vector to the curve is colinear to the fiber $\mc{D}_p$ of the odd distribution $\mc{D}$. Intuitively, the odd curve could otherwise be an integral curve of the odd distribution $\mc{D}$, which is supposed to be nonintegrable. However, there exists no such duality at the components of the Ramond divisor $\mc{R}$, because in the Ramond case, the odd distribution fails to be nonintegrable exactly along the Ramond divisor. Indeed, in coordinates $(z \, | \, \zeta)$ in which $\mc{D}$ is generated by $D_{\zeta}$ as above, the Ramond divisor is given by $z = 0$, and at every point $(0 \, | \, \zeta_0)$ of it, the vector field $D_{\zeta}$ equals $\delpzeta$, which is tangent to the Ramond divisor. Thus, the condition that the tangent space to the Ramond divisor coincides with the fiber of the odd distribution $\mc{D}$ does not determine a unique point of the Ramond divisor. This means that the term a ``Ramond \textbf{puncture}'' is actually a bit of a misnomer --- it is rather a component of the Ramond \textbf{divisor}!

\subsubsection*{Relationship to Spin Curves}

  \begin{definition}[$n$-Punctured Spin Curve \cite{jarvis2000geometry}] A \emph{family of genus $g$, $\textbf{n}$-punctured spin curves
  $(\mc{C}/T, \linebreak[0] \{s_i\}_{i=1}^n, \mc{L},j)$ of type $\tbf{m}=(m_1, \dots, m_n)$} with $m_i$ being non-negative integers 
  is the data of 
  \begin{enumerate}[$(1)$]
      \item a smooth, proper morphism $\pi: \mc{C} \to T$ of schemes of relative dimension $1$ with genus $g$ fibers $C_{t}$, 
      
      \item a sequence of $n$ distinct sections $s_i : T \to \mc{C}$, 
      
      \item a line bundle $\mc{L}$, and 
      
      \item an isomorphism $j: \mc{L}^{\otimes 2} \lgr \Omega_{\mc{C}/T}^1(m_1p_1+ \dots + m_np_n)$, where the $p_i$'s denote the images of the sections $s_i$ in $\mc{C}$. 
  \end{enumerate}
  An isomorphism of $n$-punctured spin curves is an isomorphism of curves that respects the data $(\{s_i\}_{i=1}^n, \mc{L}, j)$. 
  \end{definition}

We will now establish a dictionary between punctured spin curves and super Riemann surfaces. 
  
  \begin{theorem} \label{Rhodge} A genus $g$ super Riemann surface over an ordinary scheme $T$ with $\nr$ Ramond punctures and $\ns$ NS punctures is equivalent to a genus $g$, $(\ns + \nr)$-punctured spin curve of type $\tbf{m}=(0, \dots 0 \, | \, 1, \dots, 1)$, where the left-hand side of the parentheses denote the types on the first $\ns$  punctures and the right-hand side denotes the types on the remaining $\nr$ punctures.
  \end{theorem} 
  
  The type $\tbf{m}=(0, \dots 0\, | \, 1, \dots, 1)$ implies that the spin curve associated to a super Riemann surface has $\ns + \nr$ marked points, but that only the Ramond type points are used to \emph{twist} the spin structure. Here twisting means that $\mc{L}^{\otimes 2} \lgr \Omega^1(\mc{R}_{\bos})$, where the $\mc{R}_{\bos}$ denotes the bosonic reduction 
  of the Ramond divisor $\mc{R}$. 
  
  \begin{proof}[Proof of Theorem \ref{Rhodge}] The proof follows the same lines as that of Proposition \ref{hodge}. 
  \end{proof}

  \section{The Moduli Problem}

\begin{definition}
Let $\mnr$ denote the fibered category over $\sS$ with fiber over $T$ the groupoid $\mnr(T)$ whose objects are families $\Sigma/T$ of genus zero SUSY curves with $\nr \ge 4$ labeled Ramond punctures over $T$ and with morphisms isomorphisms of such families. \end{definition}

Let $\Sigma \in \mnr(\on{Spec} k)$. We can recover the supercurve $X$ over $k$ underlying $\Sigma$ by forgetting the superconformal structure.

\begin{proposition}
  \label{susyw}
Let $\Sigma \in \mnr(\on{Spec} k)$ and let $X$ be its underlying
supercurve. Then, $X \cong \wy = \mathbb{W}\mathbb{P}^{1|1}(1,1\; |\;
1-\nr/2)$.
\end{proposition}

\begin{proof}  
Since $X$ is split (see the proof of Lemma \ref{scm}),
the structure sequence
\begin{equation}
  \label{str-SES}
0 \to \mc{D} \to \mc{T}_X \to  \mc{T}_X / \mc{D} \to 0
\end{equation}
may be identified, after restriction to $X_{\bos}$, with a short exact
sequence of $\mc{O}_{X_{\bos}}$-modules:
\begin{equation}
\label{cansen}
0 \lr \mc{D}_{bos} \hookrightarrow \mc{T}_{X_{\bos}} \oplus \mc{J}^{\vee} \lr \mc{D}_{\bos}^{\otimes 2}(\mc{R}_{\bos}) \lr 0.
\end{equation}
Note that the arrows in the above sequence are canonical in the sense that they are part of the definition of a SUSY structure. 

Next, choose an isomorphism $\phi: \pr \lgr X_{\bos}$ and use it to pull back \eqref{cansen} to the short exact sequence of $\mc{O}_{\pp^1}$-modules
\begin{equation*} 
  0 \lr \phi^* \mc{D}_{\bos} \hookrightarrow \mc{O}_{\pr}(2) \oplus
\phi^* \mc{J}^{\vee} \lr (\phi^* \mc{D}_{\bos})^{\otimes 2}( \phi^{-1}
\mc{R}_{\bos}) \lr 0 .
\end{equation*}
By comparing the parity of the modules in the above sequence, we see
that there exist canonical isomorphims
\[
\phi^* \mc{D}_{\bos} \lgr \phi^* \mc{J}^{\vee}
\]
and
\[
(\phi^* \mc{D}_{\bos})^{\otimes 2} \lgr \mc{O}_{\pr}(2) \otimes
\mc{O}_{\pr}(-\phi^{-1} \mc{R}_{\bos}).
\]
Choosing an isomorphism $c:
\mc{O}_{\pr}(2) \otimes \mc{O}(-\phi^{-1} \mc{R}_{\bos}) \lgr
\mc{O}_{\pr}(2-\nr)$ and a square root
\[
d: \sqrt{\mc{O}_{\pr}(2-\nr)} \lgr \mc{O}_{\pr}(1-\nr/2)
\]
we get an isomorphism
\[
\phi^* \mc{J}^{\vee} \lgr \Pi \mc{O}_{\pr}(1-\nr/2).
\] 
The triple $(\phi, c ,d)$ then defines an isomorphism
\[
\psi: |\pr| \lgr |X_{\bos}|, \ \psi^{\#}: \mc{O}_X \lgr \psi_*
\left(\mc{O}_{\pr} \oplus \Pi \mc{O}_{\pr}(\nr/2 - 1)\right),
\]
which shows that $
X \cong \mathbb{W}\mathbb{P}^{1|1}(1,1\; |\; 1-\nr/2)$ by Example \ref{wsp}.
\end{proof}

\begin{corollary}
  \label{-2}
The structure distribution $\mc{D}$ and the quotient
$\mc{T}_X /\mc{D} \cong \mc{D}^{\otimes 2} (\mc{R})$ on a super
Riemann surface $\Sigma = (X, \mc{R}, \mc{D}) \in \mnr(\on{Spec} k)$
may be identified on $X \cong \wy$ as follows:
\begin{align*}
\mc{D} & \cong \Pi \mc{O}_{\wy}(1-\nr/2), \\
\mc{T}_X / \mc{D} & \cong \mc{O}_{\wy}(2).
\end{align*}
\end{corollary}

\begin{proof}
This is not so much a consequence of the statement of Proposition
\ref{susyw} but rather of its proof, in the course of which we have
identified the pullback \eqref{cansen} of the short exact sequence
\eqref{str-SES} from $X \cong \wy$ to $X_{\bos} \cong \pr$ as
isomorphic to
\[
0 \to \Pi \mc{O}_{\pr}(1-\nr/2) \to (\mc{T}_{\wy})_{\bos} \to  \mc{O}_{\pr}(2) \to 0.
\]
The required isomorphisms on $\wy$ then follow from Corollary
\ref{susybundleisontwo}.
\end{proof}

It is not true, in general, that a family of supercurves $X/T$
underlying $\Sigma/T$ is isomorphic to $\wy \times T$, even for an
ordinary scheme $T$. One way to see this is to observe that
$\Omega_{\pr \times T}^1$ does not have a unique square root and so
the proof used in Proposition \ref{susyw} does not generalize to
families. The square roots of $\Omega_{\pr \times T}^1$, in fact, form
a torsor over $H^1(T, \Z_2)$. To see this, first recall that any line
bundle on $\pr \times T$ is isomorphic to $\mc{O}_{\pr}(m) \otimes
\mc{I}$, where $\mc{I}$ a line bundle on $T$. It follows that
$\mc{O}_{pr}(m) \otimes \mc{I}$ is a square root of $\Omega_{\pr
  \times T}^1$ if and only if $m= -1$ and $\mc{I}$ is a two-torsion line bundle
on $T$.

\begin{proposition}
\label{wittenclassifying}
Any family of supercurves $X/T$ underlying $\Sigma/T$, where $T$ is an ordinary scheme, is \'etale locally on the base isomorphic to a pullback of $\wy/k$.
\end{proposition}

\begin{proof}
To prove the statement, we will show that there exists an
\'etale cover $\{U_i \to T\}_{i \in I}$ on which $X$ trivializes: $X \times_T U_i \cong \wy \times U_i$ for each $i$.

Since $X_{\bos}/T$ is an ordinary family of genus zero curves, there exists an \'etale cover $\{V_j \to T\}_{j \in J}$ and an isomorphism $X_{\bos} \times_T V_j \lgr \pr \times V_j$ for each $j \in J$.
Write $p_1: \pr \times V_j \to \pr$ and $p_2: \pr \times V_j \to V_j$ for the natural projections. Then
\[ \sqrt{\mc{O}_{\pr \times V_j}(2-\nr)} \cong p_1^* \mc{O}_{\pr}(1-\nr/2) \otimes p_2^*\mc{N}, \]
where $\mc{N}$ is a two-torsion line bundle on $V_j$. 
 
Two-torsion line bundles over $V_j$ form a $H^1(V_j, \Z_2)$-torsor and so there exists a double \'etale cover $p_j: U_j:=\on{Spec}(\mc{O}_{V_j} \oplus \mc{N}) \to V_j$ on which $p_j^*\mc{N} \cong \mc{O}_{U_j}$; here the ring structure on $\mc{O}_{V_j} \oplus \mc{N}$ is given by the $\mc{O}_{V_j}$-module structure on $\mc{N}$ and the two-torsion isomorphism $\varphi: (\mc{N})^{\otimes 2} \to \mc{O}_{V_j}$:
\[
(f,n) \cdot (f',n') := (ff' + \varphi(n \otimes n'), f n' + f' n) \qquad \text{for } f, f' \in \mc{O}_{V_j} \text{ and } n, n' \in \mc{N}.
\]

Now we can argue as in the proof of Proposition \ref{susyw} and show
that the sheaf $\mc{J}^{\vee}$ of odd nilpotents on $X \times_T U_j$
is isomorphic to $\mc{O}_{\pr \times U_j}(1-\nr/2)$, and so $X
\times_T U_j \cong \wy \times U_j$.
\end{proof}

\subsubsection*{Moduli of weighted supercurves}
\label{moduliofw}

Let $\Sigma/T$ be a family of genus $0$ super Riemann surfaces with an
even number $\nr \ge 4$ of Ramond punctures and let $X/T$ be the
underlying family of supercurves. By Proposition \ref{susyw}, the
geometric fibers $X_t$ are isomorphic to $\wb$. This is a particular
instance of a slightly more general situation.

\begin{definition} A \emph{family $X/T$ of weight $m$, genus zero supercurves} is a smooth, proper morphism $X \to T$ of relative dimension $(1|1)$ such that each fiber $X_t$ is isomorphic to $\mathbb{W} \mathbb{P}(1,1\, |\, m)$. 
\end{definition}

\begin{definition}
\label{Mm}
The \emph{moduli superstack $M(m)$ of weight $m$, genus zero supercurves} is the superstack over $\sS$ with fiber over $T$ the groupoid $M(m)(T)$ whose objects are families $X/T$ of genus zero supercurves of weight $m$ over $T$ and whose morphisms are isomorphisms of such families. 
\end{definition}

In this paper, we are interested in the case $m=1-\nr/2$ for even $\nr \ge 4$. However, our construction of $M(1-\nr/2)$ in Section \ref{construction} below could be generalized to more values of $m$. For example, if we allow for a less economical \'{e}tale cover of the base, we can prove the following generalization of Proposition \ref{wittenclassifying} to arbitrary $m$, which we will use in Section \ref{construction}.

\begin{proposition}
\label{wittenclassifying-m}
Any family of supercurves $X/T$ of weight $m$, where $T$ is an ordinary scheme, is \'etale locally on the base isomorphic to a pullback of $\wm/k$.
\end{proposition}

\begin{proof}
We will proceed as in the proof of Proposition \ref{wittenclassifying} and show that there exists an
\'etale cover $\{U_i \to T\}_{i \in I}$ on which $X \times_T U_i \cong \wm \times U_i$.

As in that proof, we start with passing to an \'etale cover $\{V_j \to T\}_{j \in J}$ which trivializes the underlying family $X_{\bos}/T$ of ordinary curves of genus zero via an
isomorphism
$\phi_j: \pr \times V_j \lgr X_{\bos} \times_T V_j$.
Then
\[
\phi_j^* \mc{J} \cong p_1^* \Pi \mc{O}_{\pr}(-m) \otimes p_2^*\mc{N} \]
for some  line bundle $\mc{N}$ on $V_j$, because $\mc{J} \cong \Pi \mc{O}_{\pr}(-m)$ on the geometric fiber $\wm$ as in Example \ref{wsp}.
 
For each $j \in J$, if we take a Zariski-open cover $\{U_{ij} \to V_j\}_{i \in I_j}$ of $V_j$ on which the line bundle $\mc{N}$ trivializes, then we can have
\[
\phi_j^* \mc{J}|_{U_{ij}} \cong p_1^* \Pi \mc{O}_{\pr}(-m)|_{U_{ij}}, 
\]
which implies that over $U_{ij}$,
the family $X/T$ is isomorphic to $\wm \times U_{ij}$.
\end{proof}

\subsection{Deformations of $\wy$} \label{deformations}

The deformation space (if it exists) of $\wy = \wb$
is the superscheme representing the functor in the following definition.

\begin{definition}
The \emph{deformation functor
\[ \on{Def}(\wy): \on{sArt}_k \to \on{Set} \]
of $\wy$} is the functor sending a local Artin $k$-superalgebra $R$ to the set of isomorphism classes of deformations of $\wy$ over $\on{Spec} R$.
\end{definition}

A \emph{first-order deformation} of $\wy$ is a deformation of $\wy$ over the spectrum of the superalgebra $k[\epsilon, \eta]/(\epsilon^2, \eta \epsilon)$, where $|\epsilon|=0$ and $|\eta|=1$. The superalgebra $k[\epsilon, \eta]/(\epsilon^2, \eta \epsilon)$ is an Artin $k$-superalgebra referred to as the \emph{super dual numbers} and plays the same role as that of the dual numbers in the classical setting.

From Lemma \ref{dimensions}, we know that the first-order deformations of $\wy$ form a $(0|\nr/2-2)$-dimensional vector space.

\begin{lemma} \label{infdef} In local coordinates $(z \, | \, \zeta)$ on $\wy$,  the vector fields \begin{equation}
\label{basisforhone}
\bigg \{ z^{-1} \delp{\zeta}, \dots, z^{-(\nr/2-2)} \delp{\zeta} \bigg \}
\end{equation}
form a basis for $H^1(\wy, \mc{T}_{\wy})$. 
\end{lemma}

\begin{proof}
To give a coordinate description of a basis for $H^1( \mc{T}_{\wy})$,  we use \v{C}ech cohomology on the cover $\mc{U}=\{U=\on{Spec} k[z  |  \zeta], V=\on{Spec} k[w | \chi]\}$ of $\wy$: 
Consider the local vector fields 
\begin{align}\label{localdescriptangent}
H^0(U, \mc{T}_{\wy} ) & = \left\{ \sum_{j \ge 0} \left(a_j + b_j \zeta \right) z^j\delp{z}  +   \sum_{j \ge 0} \left(c_j + d_j \zeta \right) z^j \delp{\zeta}\right\}, \\ 
\nonumber H^0(V, \mc{T}_{\wy}) & = \left\{ \sum_{j \ge 0} \left( \ov{a_j} + \sum \ov{b_j} \chi \right)w^j \delp{w} +   \sum_{j \ge 0} \left(\ov{c_j} w^j + \sum \ov{d_j} \chi \right) w^j \delp{\chi}\right\},
\end{align} where $a_j, d_j, \ov{a}_j, \ov{d}_j \in k$ and  $b_j= c_j= \ov{b}_j=  \ov{c}_j=0$ for the above sections to be even and vice versa for the sections to be odd and where all summations are for $j \ge 0$ and all the coefficients but finitely many vanish.  Changing coordinates on $U \cap V$ we compute: 
\begin{align*} 
\delp{w} = & -z^2 \delp{z} +  (\nr/2-1)\zeta z  \delp{\zeta},\\ 
\delp{\chi} = & z^{1-\nr/2} \delp{\zeta}.
\end{align*} 

A basis for the super vector space $H^1(\mc{T}_{\wy})$ is given by a basis for the nullspace of the following system of equations: 
\begin{align} \label{defining}  
\sum_{j \ge 0} a_j z^j  - \sum_{j \ge 0} \ov{a}_{j} z^{-j+2}= & 0,  & :\frac{\partial}{\partial z}  \\
\nonumber
    \sum_{j \ge 0} b_j z^j - \sum_{j \ge 0} \ov{b}_{j} z^{\nr/2+1-j} = & 0,  &  :\zeta \frac{\partial}{\partial z}   \\
    \nonumber
\sum_{j \ge 0} c_j z^j  - \sum_{j \ge 0} \ov{c}_{j} z^{1-j-\nr/2} = & 0, & : \delp{\zeta} \\
\nonumber
\sum_{j \ge 0} d_j z^j - (\nr/2-1) \sum_{j \ge 0} \ov{a}_{j} z^{-j+1} - \sum_{j \ge 0} \ov{d}_{j} z^{-j}= & 0, & : \zeta \delp{\zeta},
\end{align}
We find that the following vector fields form a basis for $H^1(\mc{T}_{\wy})$: 
\begin{equation*}
\bigg \{ z^{-1} \delp{\zeta}, \dots, z^{-(\nr/2-2)} \delp{\zeta} \bigg \}. \qedhere
\end{equation*}
\end{proof}

This computation implies that a first-order deformation of $\wy$ is uniquely determined by its local trivializations $U \linebreak[3]= \linebreak[2] \on{Spec} k[z, \linebreak[0] \epsilon \linebreak[1] \, | \, \zeta, \eta]$ and $V=\on{Spec} k[w, \epsilon \, | \, \chi, \eta]$ and its gluing data:
\begin{align} \label{gluingfunction}
 w & \mapsto 1/z, \\
\nonumber \chi & \mapsto z^{\nr/2-1} \left( \zeta  + \eta \sum_{i=1}^{\nr/2-2} c_i z^{-i} \right) = \zeta z^{\nr/2-1} + \eta \sum_{i=1}^{\nr/2-2} c_i z^{\nr/2- 1 -i} \end{align}
with $c_i \in k$.

Let
\[
S := \mathbb{H}^1 (\mc{T}_{\wy}) : = \on{Spec} \mc{S}( H^1(\mc{T}_{\wy})^{\vee})
\]
be the affine superspace based on the super vector space $H^1(\mc{T}_{\wy})$ and choose coordinates $\eta_1, \dots, \linebreak[0] \eta_{\nr/2-2}$ corresponding to the basis of $H^1(\mc{T}_{\wy})$ given in Lemma \ref{infdef}. The superspace $S$ is then identified with the affine superspace
\[
S \cong \mathbb{A}^{0|\nr/2-2} = \on{Spec} k [\eta_1, \dots, \eta_{\nr/2-2}].
\] 
We will denote the supersymmetric algebra $\mc{S} (H^1(\mc{T}_{\wy})^\vee)$ by $B$, so that
\begin{align*}
    B & = \mc{S} (H^1(\mc{T}_{\wy})^\vee) \cong k [\eta_1, \dots, \eta_{\nr/2-2}] \text{ and}\\
S & = \mathbb{H}^1 (\mc{T}_{\wy}) = \Spec B.
\end{align*}
We denote by $Z$ the deformation of $\wy$ over $S$ with local trivialization $U=\on{Spec}B[z\, | \, \zeta]$ and $V=\on{Spec}B[w\, | \, \chi]$ and gluing data, 
\begin{align} \label{gluingformula1} 
 w & \mapsto 1/z, \\
\nonumber \chi & \mapsto \zeta z^{\nr/2-1} + \sum_{i=1}^{\nr/2-2} \eta_i z^{\nr/2-1 - i}, \end{align}
or, equivalently,
\begin{align} 
\label{gluingformula} 
 z & \mapsto 1/w, \\
\nonumber \zeta & \mapsto \chi w^{\nr/2-1} - \sum_{i=1}^{\nr/2-2} \eta_i w^{i}. \end{align}

 \begin{theorem} \label{versal} The functor $\deform(\wy)$ is represented by the superscheme $S$ with universal deformation $Z$.
 \end{theorem}

\begin{proof}
Let $X \in \on{Def}(\wy)(R)$, where $R \in \on{sArt}_k$. The surjection $R \to R/\mc{J}=R_{\bos}$ can be factored into a finite composition
\[ R= R_m \to R_{m-1} \to \dots \to  R_1 \to R_0 = R_{\bos} \]
of small extensions by $\mc{J}^i/\mc{J}^{i+1}$, where $\mc{J}$ is the ideal of odd nilpotents in $R$. Denote the pullback of $X$ to $R_i$ by $X_i$.

Since $\on{dim}_k H^1(\wy, \mc{T}_{\wy})^+=0$, $\wy$ has no nontrivial even deformations, so, in particular, there exists no nontrivial deformations over the spectrum of a purely even local Artin $k$-algebra. Thus, any deformation of $\wy$ over $\Spec R_0$ is trivial, \emph{i.e}., $X_0 \cong \wy \times \Spec R_{0} = Z \times_{S, \tau} \Spec R_{0}$ for the canonical map $\tau: S \to k \to R_{0}$, which sends all the $\eta_j$'s to 0.

The set of isomorphism classes of deformations of $X_0$ over $R_1$ is a torsor over the group $H^1(X_0, \mc{T}_{X_0/ R_0}) \linebreak[0] \otimes_{R_0} \mc{J}/\mc{J}^2 = H^1(\wy, \mc{T}_{\wy}) \linebreak[0] \otimes_k \mc{J}/\mc{J}^2$.
Note that we always have a ``trivial'' deformation  $Z \times_{S, \tau} \Spec R_1 \cong \wy \times \Spec R_1$ of $X_0 \cong \wy \times \Spec R_0$, which we can use as a base point of the torsor of deformations of $X_0$ over $R_1$ and identify this torsor with the group $H^1(\wy, \mc{T}_{\wy}) \linebreak[0] \otimes_k \mc{J}/\mc{J}^2$. This pullback $Z \times_{S, \tau} \Spec R_1$ is covered by the affine charts $U' \cong \on{Spec} R_1[z \, | \,\zeta]$ and $ V' \cong \on{Spec} R_1[w \, | \,\chi]$ with transition functions
\begin{align}
\label{undeformed}
w & \mapsto 1/z, \\
\nonumber \chi & \mapsto \zeta z^{\nr/2-1}.
\end{align} 
From the computation of $H^1(\mc{T}_{\wy})$ in Lemma \ref{infdef}, we find that the elements of $H^1(\mc{T}_{\wy}) \linebreak[0] \otimes_{k} \mc{J}/\mc{J}^2$ may be identified with linear combinations
\begin{equation}
\label{vfield}    
\sum_{j=1}^{\nr/2-2} d_j z^{-j} \delp{\zeta}
\end{equation}
of the vector fields in \eqref{basisforhone} with coefficients $d_j \in \mc{J}/\mc{J}^2$.

Given the formulas \eqref{undeformed} describing the trivial deformation, the deformation which corresponds to an element \eqref{vfield} is described, up to isomorphism, by the following gluing functions
\begin{align*} 
w & \mapsto 1/z, \\
\nonumber \chi & \mapsto \zeta z^{\nr/2-1} + \sum_{j=1}^{\nr/2-2} d_j z^{\nr/2-1-j}.
\end{align*} 
This means that there exist $d_j \in \mc{J}/\mc{J}^2$, $j = 1, \dots, \nr/2-2$, such that the deformation $X_1$ is isomorphic to the one given by these gluing functions. 
Define $f_1: B \to R_1: \eta_j \mapsto d_j$, then $X_1 \cong Z \times_{S,f_1} \Spec R_1$ as deformations of $X_0$ over $R_1$. 

Now continue inductively on $i$, starting from $i=1$ to move up to $i = m$. Assume we have defined $f_i: B \to R_i$ such that $X_i \cong Z \times_{S,f_i} \Spec R_i$ as a deformation of $X_{i-1}$ (and therefore, $\wy$) over $\Spec R_i$.

Since $B$ is a free $k$-superalgebra, there exists a map $\tau: B \to R_{i+1}$ restricting to $f_i: B \to R_i$.

Then $Z \times_{S,\tau} R_{i+1}$ is an infinitesimal deformation of $X_{i}$ over $R_{i+1}$, which is glued along the intersection of the affine charts $U' \cong \on{Spec} R_{i+1}[z \, | \,\zeta], V' \cong \on{Spec} R_{i+1}[w \, | \,\chi]$ by the transition functions
\begin{align*} 
w & \mapsto 1/z,\\
\nonumber \chi & \mapsto \zeta z^{\nr/2-1} + \sum_{j=1}^{\nr/2-2} \tau(\eta_j) z^{\nr/2-1-j}.
\end{align*} 

Since set of isomorphism classes of deformations of $X_i$ over $R_{i+1}$ is a torsor over $H^1(X_i, \mc{T}_{X_i/ R_i}) \linebreak[0] \otimes_{R_i} \mc{J}^{i+1}/\mc{J}^{i+2} $, which is isomorphic to $ H^1(\wy, \mc{T}_{\wy}) \otimes_{k} \mc{J}^{i+1}/\mc{J}^{i+2}$ because of the base change theorem, cf.\ \cite[Theorem 8.5.9 (b) and Remark 8.5.10 (b)]{FGA}, with a base point given by $\tau$ just above, there exist $d_j \in \mc{J}^{i+1}/\mc{J}^{i+2}$ for $j = 1, \dots, \nr/2-2$ such that
\begin{align*} 
w & \mapsto 1/z, \\
\nonumber \chi & \mapsto \zeta z^{\nr/2-1} + \sum_{j=1}^{\nr/2-2} (\tau(\eta_j) + d_j) z^{\nr/2-1-j}
\end{align*} 
are the gluing functions of $X_{i+i}$. 
Define $f_{i+1}: R \to R_{i+1}$ to send $\eta_j \to \tau(\eta_j) + d_j$. Then $Z \times_{S,f_{i+1}} R_{i+1} \cong X_{i+1}$. This completes the induction step.
\end{proof}

\subsubsection*{Projective embeddings of $Z$}

We will look at couple of useful closed embeddings of $Z$ into weighted projective superspaces over $S$. Let us start with general definitions.

We say that a superscheme $\rho: X \to T$ over a superscheme $T$ is \emph{projective} if $X$ admits a closed immersion into a super projective bundle $\Proj \mc{S}(\mc{E})$ for some quasi-coherent, finite-type sheaf $\mc{E}$ over $T$. A line bundle (\emph{i.e}., invertible sheaf) $\mc{L}$ over $X$ is \emph{very ample relative to} $T$, if there is an immersion $i: X \to \Proj \mc{S}(\mc{E})$ over $T$ such that $\mc{L} \cong i^* \mc{O}_{\Proj \mc{S}(\mc{E})} (1)$ for a quasi-coherent sheaf $\mc{E}$ on $T$. Under certain assumptions, see the proof of Proposition \ref{zisprojective}.2, we will have $\mc{E} \cong \rho_* \mc{L}$,
the sheaf $\mc{E}$ will be locally free, the immersion $i: X \to \Proj \mc{S}(\mc{E})$ will be associated with the line bundle $\mc{L}$, and $X$ will be projective over $T$.

\begin{proposition} \label{zisprojective} 
\begin{enumerate}
    \item 
The supercurve $Z/S$ is a closed subsuperscheme of $\mathbb{W} \mathbb{P}^{1|2}_S(1, 1 \, |\,  0, 0) = \mathbb{P}^1_S \times_S \mathbb{A}^{0|2}_S \cong \mathbb{P}^1 \times \mathbb{A}^{0|\nr/2}$ over $S$. 
\item 
The line bundle $\mc{O}_Z(1)$ defines an embedding of $Z$ as a closed subsuperscheme into the projective superspace $\pp_S^{1\, | \, \nr/2 + 1} = \pp^{1\, | \, \nr/2 + 1} \times S$ over $S$.
\end{enumerate}
\end{proposition}

\begin{proof} 
1. From the gluing description \eqref{gluingformula1} of the deformation $Z$ of $\wy$ over $S$ with trivializing cover $\mc{U}=\{ U=\on{Spec}B[z | \zeta], V = \on{Spec}B[w | \chi]\}$, 
and patching together along  $U \cap V$ via the automorphism of $Z \vert_{U \cap V}$ sending
\begin{align}
w & \mapsto 1/z, \\
\nonumber \chi & \mapsto \zeta z^{\nr/2-1} + \sum_{i=1}^{\nr/2-2} \eta_i z^{\nr/2-1-i}, \end{align} 
we find that
\begin{align}
\chi & = \zeta (v/u)^{\nr/2-1} + \sum_{i=1}^{\nr/2-2} \eta_i (v/u)^{\nr/2-1-i}, \\
\nonumber u^{\nr/2-1} \chi & = v^{\nr/2-1} \zeta + \sum_{i=1}^{\nr/2-2} \eta_i u^i v^{\nr/2-1-i}.
\end{align}
Therefore, 
\[
Z \cong \on{Proj} B[u,v | \zeta, \chi]/ ( u^{\nr/2-1} \chi - v^{\nr/2-1} \zeta - \sum_{i=1}^{\nr/2-2} \eta_i u^i v^{\nr/2-1-i}),
\] 
where the even coordinates $u,v$ have degree one and the odd coordinates $\zeta, \chi$ have degree zero. 
The proof of the proposition then follows from the identification of $\on{Proj}B[u,v|\zeta, \chi]$ with $\mathbb{P}^1_S \times_S \mathbb{A}_S^{0|2}$.
\smallskip

\noindent
2. To see that $Z/S$ is projective and the sheaf $\mc{O}_Z(1)$ is very ample relative to $\pi: Z \to S$, 
we will use Codogni's criterion \cite[Criterion 3]{codogni}, based on an argument of \cite[Theorem 1]{LBPW1990proj}:
\begin{quote}
    Suppose $\rho: X \to T$ is a supercurve over a smooth base $T$ and $\mc{L}$ is a rank-$(1|0)$ line bundle over $X$ such that $\rho_* \mc{L}$ is locally free on $T$. Let $\mc{J} = \mc{J}/\mc{J}^2$ be the odd-function ideal of the pullback $X_{\on{split}} := X \times_T T_{\bos}$ of $X$ to $T_{\bos} \hookrightarrow T$. The supercurve $X_{\on{split}}$ over $T_{\bos}$ will be split with structure sheaf $\mc{O}_{X_{\bos}} \oplus \mc{J} =  (\mc{O}_{X_{\on{split}}}/\mc{J}) \oplus \mc{J}$. If, for every pair of closed points $x, y$ of $X_{\bos}$ such that $\rho^{\bos}(x) = \rho^{\bos}(y)$, we have
\[    
    R^1 \rho^{\bos}_* (\mc{L}_{\bos} \otimes I_x \otimes I_y) = 0
\]
and
\[    
    R^1 \rho^{\bos}_* (\mc{L}_{\bos} \otimes \mc{J} \otimes I_x) = 0,
\]
where $I_z$ is the ideal sheaf of a closed point $z$ of $X_{\bos}$, then $\mc{L}$ defines a closed immersion $X \to \Proj \mc{S}(\rho_*\mc{L})$ to a relative projective superspace over $T$. In particular, $X$ will be relatively projective and $\mc{L}$ will be relatively very ample.
\end{quote}

To apply this criterion, we need to check that the sheaf $\pi_* \mc{O}_Z(1)$ is locally free. In fact, it is locally free of the same rank $(2\, | \, \nr/2+1)$ as the dimension of the super vector space $H^0 (\wy, \mc{O}_{\wy}(1))$, because of the standard argument of ``cohomology and base change'', given that $H^1 (\wy, \mc{O}_{\wy}(1)) = 0$. The reason for this vanishing is identification $\mc{O}_{\wy} (1)_{\bos} \cong \mc{O}_{\pp^1}(1)$ and $\mc{J}(1) \cong \mc{O}_{\pp^1}(\nr/2)$ by \eqref{wm-functions}.  Then, in our case, $S_{\bos} = \Spec k$, $Z_{\on{split}} = \wy$, and $Z_{\bos} = \pp^1$ with $\mc{O}_Z (1)_{\bos} \cong \mc{O}_{\pp^1}(1)$ and $\mc{J} \cong \mc{O}_{\pp^1}(\nr/2-1)$. Thus, $\mc{O}_Z (1)_{\bos} \otimes I_x \otimes I_y \cong \mc{O}_{\pp^1} (-1)$ and $\mc{O}_Z (1)_{\bos} \otimes \mc{J} \otimes I_x \cong \mc{O}_{\pp^1} (\nr/2-1)$. The first cohomology group of both shaves vanishes, and Codogni's criterion applies.

Finally,  the local triviality of $\pi_* \mc{O}_Z(1)$ implies its triviality, because $S$ has just one point. Hence, $\Proj \mc{S}(\pi_* \mc{O}_Z(1)) \cong \pp_S^{1\, | \, \nr/2 + 1}$.
\end{proof}

The proposition may be visualized according to the following diagram: 
\[
\begin{tikzcd}
Z \arrow[rd] \arrow[rr, hook] &                                             & {\mathbb{W} \mathbb{P}^{1|2}_S(1, 1 \, |\,  0, 0) = \pp H^0(Z,\mc{O}_Z(1))}  \arrow[ld] \\
                              & S=\operatorname{Def}(\mathbb{W} \mathbb{P}) &                                                              
\end{tikzcd}
\]

\subsection{The moduli superstack $M(1-\nr/2)$ of supercurves}
\label{construction}

In this section, we will prove that the moduli superstack $M(1-\nr/2)$ of supercurves of weight $1-\nr/2$ is algebraic and isomorphic to the quotient superstack $[S/\awy]$. Here and henceforth, set
\[
M:= M(1-\nr/2)
\]
to be the moduli superstack of genus zero supercurves of weight $1-\nr/2$, see Definition \ref{Mm}.

\subsubsection*{A formally smooth cover of $M(1-\nr/2)$}

The first step will be to construct a certain morphism $S \to M = M(1-\nr/2)$ and prove that it is formally smooth (and, in fact, a smooth cover). 
Using the 2-Yoneda lemma \cite[Section 3.2]{olsson2016algebraic}, we can identify the groupoid $M(T)$ (for any superscheme $T$) with the groupoid of superstack morphisms $T \to M$, up to canonical equivalence. In particular, the object $Z \in M(S)$ constructed earlier corresponds to a certain morphism $S \to M$.

\begin{definition}
Let $U$ be a superscheme and $N$ a superstack. We say that
a superstack morphism $U \to N$ is \emph{formally smooth}, if the following condition holds:
\begin{quote} 
For every affine superscheme $T$ with a morphism $T \to N$ and every closed immersion $T_0 \hr T$ defined by a nilpotent ideal,
there exists a map (not necessarily unique) filling in the dotted arrow of the diagram 
\begin{equation}\label{smoothness}
\begin{tikzcd}
T_0 \arrow[d, hook] \arrow[r]     & U \arrow[d] \\
T \arrow[r] \arrow[ru, dotted] & N
\end{tikzcd}
\end{equation} 
and making the resulting two triangles commute. 
\end{quote} 
\end{definition}

Applying this definition to the superstack morphism $S \to M$, observe that since $T$ is a nilpotent extension of $T_0$, the supercurve associated to the morphism $T \to M$ is a deformation of the supercurve associated to the morphism $T_0 \to T \to M$.

\begin{theorem} \label{formalsmooth} The morphism $ S \to M(1-\nr/2)$ is formally smooth. 
\end{theorem}

\begin{proof}
Given a diagram
\begin{equation*}
\begin{tikzcd}
T_0 \arrow[d, hook] \arrow[r]     & S \arrow[d] \\
T \arrow[r]  & M,
\end{tikzcd}
\end{equation*}
take $X/T$ and $X_0/T_0$ to be the supercurves associated to the maps $T \to M$ and $T_0 \to T \to M$, respectively. We may assume that $T$ is affine and $T_0 \hr T$ is a closed immersion defined by a square-zero ideal. Indeed, we can factor an immersion given by an arbitrary nilpotent ideal into a sequence of immersions $T_i$ defined by square-zero ideals, as we did in the proof of Theorem \ref{versal}. Since the induction step will be starting with a trivial deformation $X_i \cong \wy \times (T_i)_{\bos}$, there will be no difference in treating the induction step for $i=0$ and $i \ge 1$, as opposed to the argument of Theorem \ref{versal}.

We first consider the case when $T$ is an ordinary affine scheme. Since $T_0$ has no odd functions, the map $T_0 \to S$ factors as
\[ T_0 \to S_{\bos}=\on{Spec} k \to S, \]
whence we are left with the diagram
\begin{equation*}
\begin{tikzcd}
T_0 \arrow[d, hook] \arrow[r]    & \on{Spec} k \arrow[d] \\
T \arrow[r] \arrow[ru, dotted, "t"] & M,
\end{tikzcd}
\end{equation*}
where $t: T \to \on{Spec} k$ is the unique map from $T$ to the point. 

To the map $\Spec k \to M$ is associated the object $\wy \in 
M(\Spec k)$. This is because we constructed $Z$ as an odd deformation of $\wy$. Moreover, we have $X_0 \cong \wy \times T_0$, because $T_0 \to M$ factors through $\Spec k$.  Our goal now is to show that $X \cong \wy \times T$ as this will immediately imply that the map $t$ fills in the dotted arrow of the above diagram: indeed,
the pullback of $Z \vert_{\on{Spec} k}$ to $T$ by $t$ is isomorphic to $\wy \times T$.

Since $T_0 \to T$ is an even extension of $T_0$, meaning that the ideal defining $T_0$ contains no odd elements, the supercurve $X$ must be isomorphic to an even deformation of $\wy \times T_0$ over $T$. This means that we can identify $X$ with a class in $H^1(\wy \times T_0, \mc{T}_{\wy \times T_0/T_0})^+$, given that the trivial deformation may be used as a base point of the deformation torsor. Since $H^1(\wy \times T_0, \linebreak[0] \mc{T}_{\wy \times T_0/T_0})^+ \linebreak[1] = 0$, the superscheme $\wy \times T_0$ has no nontrivial deformations over $T$. In particular, we must have $X \cong \wy \times T$. 

Let us now turn to the general case when $T$ is an affine superscheme. 
Consider the diagram
\begin{equation*}
\begin{tikzcd}
(T_0)_{\bos} \arrow[r, hook] \arrow[d, hook] & T_0 \arrow[d, hook] \arrow[r] & S \arrow[d]  \\
\sbs \arrow[r, hook] & T \arrow[r] 
& M.        
\end{tikzcd}
\end{equation*}

Since $T$ is Noetherian, the odd-nilpotent ideal $\mc{J}_{T}$ defining $\sbs \hookrightarrow T$ is nilpotent, and we may assume that $\mc{J}_{T}^2=0$. This is because we can always factor $\mc{J}_{T}$ into a finite composition of square-zero ideals, as in the proof of Theorem \ref{versal}. Furthermore, since
\[ X \times_{T} \sbs \cong \wy \times \sbs, \]
we can treat $X/T$ as a deformation of $\wy \times \sbs$ over $T$, and thus we can identify $X/T$ with a class $[X] \in
H^1( \wy \times \sbs, \mc{T}_{\wy \times \sbs/\sbs})$. 

We can read off the 
gluing functions of
$X/T$ by representing the class $[X]$ by a \v{C}ech co-cycle 
\[
\sum_{i=1}^{\nr/2-2} d_i z^{-i} \frac{\partial}{\partial z}
\]
in the standard two charts of $\wy$, again as in the proof of Theorem \ref{versal}:

\begin{align*} 
w & \mapsto 1/z \\
\nonumber \chi & \mapsto \zeta z^{\nr/2-1} + \sum_{i=1}^{\nr/2-2} d_i z^{\nr/2-1-i}. 
\end{align*} 
where $d_i  \in \mc{J}_{T}$. 

Then the map $f: T \to S$ defined by $\eta_i \mapsto d_i$ for all $i$ will pull the family $Z$, defined by the gluing functions \eqref{gluingformula1}, to $X$. Therefore, $f$ will fill in the dotted arrow of the commutative diagram \eqref{smoothness}. 
\end{proof}

\subsubsection*{Representability of the diagonal}

Here we will show that the diagonal morphism \[ \Delta: M \to M \times M \] is representable by superschemes, \emph{i.e}., for any superscheme $T$ and morphism $T \to M \times M$, the pullback $T \times_{M \times M} M$ is a superscheme, see Corollary \ref{diagonalrep}. 

The morphism $T \to M \times M$, as a morphism of superstacks, can be identified with a pair of objects $X_1, X_2 \in M(T)$ and an isomorphism $\sigma: X_1 \lgr X_2$ in $M(T)$. The functor
\[
\Isom_T (X_1, X_2): \sS/T \to \on{Set}
\]
sending $T' \to T$ to the set $\on{Isom}_{T'}(X_1 \times_T {T'}, X_2 \times_T {T'})$ of $T'$-isomorphisms, see \cite[3.4.7]{olsson2016algebraic}, is a sheaf on $\sS/T$ isomorphic to the pullback $T \times_{M \times M} M$. We, therefore, must show that the sheaf $\Isom_T(X_1, X_2)$ is representable by a $T$-superscheme.

Let us first address a particular case of this for $T = S$ and $X_1 = X_2 = Z$. Then we will show that the general case follows.

\begin{theorem}
\label{aut-rep}
The functor $\on{Aut}_S(Z): \sS/S \to \on{Group}$ sending $T \to S$ to the group of automorphisms of $Z \times_{S} T$ over $T$ is representable. 
\end{theorem}

\begin{proof}

Note that when we talk about an automorphism, we will mean a $T$-point of $\on{Aut}_S(Z)$ for a ``probe'' superscheme $T$ over $S$.

Recall that $Z$ may be embedded into $\pp^{1\, |\, \nr/2+1}_S$, so that the pullback $\mc{L}$ of $\mc{O}(1)$ from $\pp^{1\, |\, \nr/2+1}_S$ is isomorphic to $\mc{O}_Z (1)$, see Proposition \ref{zisprojective}.2.

Just as in the proof of Theorem \ref{rep}, locally on $T$, the line bundle $\mc{O}_Z(1)$ becomes invariant, up to isomorphism, under a given automorphism $\alpha$ of $Z/S$, \emph{i.e}., $\alpha^* \mc{O}_Z(1) \cong \mc{O}_Z(1)$, and so does the group $H^0 (Z, \mc{O}_Z(1))$. Therefore, the automorphism $\alpha$ induces an $S$-automorphism of the homogeneous ring $\bigoplus_{m \in \Z} H^0 (Z, \mc{O}_Z(m))$ of $\pp^{1\, |\, \nr/2+1}_S$, and this automorphism gives an $S$-automorphism of the ambient $\pp^{1|\nr/2+1}_S$. This implies that the automorphisms of $Z/S$ extend to automorphisms of $\pp_S^{1 | \nr/2+1}/S$, and hence, $\Aut _S (Z)$ may be identified with a subquotient of $\on{Aut}_S  (\mathbb{P}_S^{1 \,| \, \nr/2+1})$. The functor $\on{Aut}_S(Z)$ is, therefore, representable by a subquotient of  $\on{Aut}_S (\mathbb{P}_S^{1 \, |\, \nr/2+1})$, which is an algebraic group superscheme as per \cite{fioresi2015projective, fioresi2021quotients}. 
\end{proof}

\begin{corollary}
\label{diagonalrep}
The diagonal map $\Delta: M \to M \times M $ is representable by superschemes.
\end{corollary}

\begin{proof}
What we need to prove is that, for every superscheme $T$, the sheaf $\, \Isom_T(X_1, X_2)$ on $\sS/T$ is representable by a $T$-su\-per\-scheme.

It follows from Proposition \ref{wittenclassifying-m} that there exists an \'etale cover $\{U_i\}_{i \in I}$ of $\sbs$ on which 

$X_2$ trivializes:
$X_2 \times_T U_i \cong \wy \times U_i$. This \'etale cover extends uniquely to an \'etale cover $\{U'_i \to T\}$. We showed in Theorem \ref{versal} that for a  choice of $U_i$-isomorphism $X_2 \times_T U_i \lgr \wy \times U_i$, there exists a map $f_i: U_i' \to S$ such that $X_2 \times_T U_i' \cong Z \times_{S,f_i} U_i'$ over $U_i'$. Recall that the map $f_i$ comes from treating $X_2 \times_T U_i' $ as a deformation of $\wy \times U_i$ over $U_i'$.  

The set of $U_i'$-isomorphisms $X_1 \times_T U_i' \lgr X_2 \times_T U_i'$ is naturally a torsor over the group of $U_i'$-automorphisms of $X_2 \times_T U_i'$. (It is also a torsor over the group of $U_i'$-automorphisms of $X_1 \times_T U_i'$.)

Since $X_2 \times_T U_i' \cong Z \times_{S,f_i} U_i'$, it follows that $\Isom_{U_i'}(X_1 \times_T U_i', X_2 \times_T U_i')$ is a torsor over $\on{Aut}_{U_i'}(Z \times_{S,f_i} U_i')$. We know the latter functor is representable (because $\on{Aut}_S(Z)$ is representable by Theorem \ref{aut-rep}). 
This shows $\Isom_T(X_1, X_2)$ to be \'etale locally isomorphic to the supergroup scheme $\on{Aut}_{U_i'}(Z \times_{S,f_i} U_i')$ and, therefore, a superscheme by \'etale descent.
\end{proof}

\subsubsection*{A smooth cover of $M(1-\nr/2)$}

Since the diagonal map $\Delta: M \to M \times M $ is representable, any morphism $T \to M$ is representable by superschemes. In particular, the formally smooth morphism $S \to M$ is representable. Since $S$ is finite-type (in fact, finite), the morphism $S \to M$ is smooth.

\begin{proposition}
The map $S \to M$ identified with $Z \in M(S)$ is a smooth cover. 
\end{proposition} 

\begin{proof} We have already showed that $S \to M$ is smooth. The morphism  $S \to M$ is surjective  if for every $T \to M$, the map $S \times_M T \to T$ is surjective. It is enough to check surjectivity on closed points, \emph{i.e}., that
\[
(S \times_M T)(\on{Spec} k) \to T(\on{Spec} k)
\]
is surjective, or equivalently, that
\begin{equation}
\label{surjectie}
(S_{\bos} \times_{M_{\bos}} \sbs)(\on{Spec} k) \to \sbs(\on{Spec} k)
\end{equation}
is surjective.  

Note that $S_{\bos} \times_{M_{\bos}} \sbs$ is \'etale locally isomorphic to the group superscheme $\awy_{\bos} \times \sbs$ by Proposition \ref{wittenclassifying-m} and identification of $\Isom$ as in the proof of Corollary \ref{diagonalrep}. Since the $\sbs$-morphism $\awy_{\bos} \times \sbs \to \sbs$ is obviously surjective, \eqref{surjectie} is surjective as well. 
\end{proof} 

\subsubsection*{Action of $\awy$ on $S$ and $Z$}

The group superscheme $\awy$ acts naturally on $S$, which is the affine superspace $\mathbb{H}^1(\wy, \mc{T}_{\wy})$ associated with the super vector space $H^1(\wy, \mc{T}_{\wy})$. More explicitly, if  $g$ is a $T$-point of $\awy$ and $\eta$ is a global vertical vector field on $(U \cap V) \times T \to T$
representing a cohomology class $[\eta] \in H^1(\wy, \mc{T}_{\wy}) \otimes \mc{O}_T^+$, then $[\eta'] := g \cdot [\eta]$ is the class represented by the
vertical vector field $g_* (\eta) \circ g^{-1}$
on $(U' \cap V') \times T \to T$, where $U' = gU$, $V' = gV$, and $g_*: \mc{T}_{(U \cap V) \times T/T} \to g^* \mc{T}_{(U' \cap V') \times T/T}$ is the pushforward (differential) map.

This action lifts to the universal deformation $Z$ of $\wy$. Since we defined $Z$ by gluing it from two charts, let us describe this action in charts. Suppose that $g$ is a $T$-point of $\awy$. Then $g$ maps the first chart $U$ on $\wy$ with coordinates $(z \, | \,\zeta)$ to a chart $U' = gU$ with coordinates $(z' \, | \, \zeta') = g \cdot (z \, | \,\zeta)$. Similarly, $g$ maps the second chart $V$ on $\wy$ with coordinates $(w \, | \,\chi)$ to a chart $V' = gV$ with coordinates $(w' \, | \, \chi') = g \cdot (w \, | \,\chi)$. Remember that $Z$ was obtained by gluing the charts $U$ and $V$ via an isomorphism $(w,\chi) = (w(z), \chi(z,\zeta, \eta))$ given by \eqref{gluingformula1}. We define the action of $g$ on $U \subset Z$ as $g \cdot (z \, | \,\zeta) = (z',\zeta')$.
Then we define $g$ to act on $V \subset Z$ as $g \cdot (w \, | \,\chi) = (w(z'), \chi (z', \zeta', \eta'))$, where $\eta' = g \cdot \eta$ is the action of $g$ on a $T$-point $\eta$ of $S = \mathbb{H}^1(\wy, \mc{T}_{\wy})$ described in the previous paragraph.

\subsubsection*{$M(1-\nr/2)$ as a quotient superstack $[S/\awy]$}

In this section we will prove that $M(1-\nr/2) = [S/\awy]$. The proof will require the following technical lemma, 

\begin{lemma} \label{technicallemma} Consider the diagram of algebraic superstacks, 
\begin{equation} \label{diagram}
\begin{tikzcd}
D \arrow[r] \arrow[d] & E \arrow[r] \arrow[d] & F \arrow[d] \\
A \arrow[r]           & B \arrow[r]           & C   
\end{tikzcd}
\end{equation}
where $A \to B$ and $F \to C$ are smooth covers. Suppose $D \cong A \times_C F$. Then  $D \cong A \times_B E$ if and only if $E \cong B \times_C F$. 
\end{lemma} 

\begin{remark}
The lemma and its proof are valid in any category with fiber products.
\end{remark}

\begin{proof}
(1) If $E \cong B \times_C F$, then $A \times_C F = A \times_B B \times_C F \cong A \times_B E$ and therefore $D \cong A \times_C F$.

(2) For the converse statement,
we will show that the morphism $\phi:E \to B \times_C F$ in the commutative diagram 
\begin{equation}
\label{big-four}
\begin{tikzcd}
D \arrow[d, equal] \arrow[r] \arrow[r] \arrow[r] & E \arrow[r] \arrow[d, "\phi"]    & F \arrow[d, equal] \\
D \arrow[r] \arrow[d]                                          & B \times_C F \arrow[r] \arrow[d] & F \arrow[d]                      \\
A \arrow[r]                                                    & B \arrow[r]                      & C                           
\end{tikzcd}
\end{equation}
is an isomorphism.

It follows from the assumption $D \cong A \times_C F$ that $D \cong A \times_B (B \times_C F)$. Thus, given that $D \cong A \times_B E$, we have that $D \cong D \times_{B \times_C F} E$ by Part 1 applied to the left $2 \times 1$ rectangle in \eqref{big-four}. Since $A \to B$ is a smooth cover and so is $D \to B \times_C F$, whereas $D \xrightarrow{\id} D$ is an isomorphism, $\phi$ must be an isomorphism by the stability of isomorphisms under smooth morphisms.
\end{proof}

\begin{theorem} \label{artin} $M=[S/\awy]$. 
\end{theorem}

\begin{proof} 
Let us first define a morphism
\[
F: [S/\awy] \to M
\]
of superstacks. The idea is that the $\awy$-equivariant morphism $Z \to S$ from the previous section (\textbf{Action of $\awy$ on $S$}) induces a superstack morphism $[Z/\awy] \to [S/\awy]$, which is a family of supercurves over the quotient superstack $[S/\awy]$. This yields a canonical morphism of superstacks $F: [S/\awy] \to M$. To avoid talking about families of supercurves over superstacks, let us give a direct argument using the functor of points of a superstack, \emph{i.e}., using the definition of a superstack as a category fibered in groupoids over the category of superschemes.

The algebraic superstack $[S/\awy]$ is the stackification of the \emph{action groupoid} $\{ S/\awy \}$. This is the category fibered in groupoids assigning to each superscheme $T$ the groupoid
whose objects are $T$-points $f: T \to S$ of $S$ and morphisms $f \to f'$ are given by elements $g \in \awy(T)$ such that $f' =  g \cdot f$,
where $g \cdot f$ denotes the action of $\awy(T)$ on $S(T)$ described in the previous section. Each object $f \in S(T)$ 
gives the supercurve $Z \times_{S,f} T$, and each morphism $g: f \to f' = g \cdot f$ lifts to a $T$-isomorphism $Z \times_{S,f} T \xrightarrow{g}  Z \times_{S,f'} T$ by the lifting of the action of $\awy$ from $S$ to $Z$ in the previous section. This defines a natural morphism of algebraic superstacks
\begin{align*}
F(T): [S/\awy](T) &\to M(T), \\ 
f &\mapsto Z \times_{S,f} T, \\ 
(g: f \to f') &\mapsto \left( Z \times_{S,f} T \xrightarrow{g} Z \times_{S, f'} T \right).
\end{align*} 

To prove that $F$ is an isomorphism of algebraic superstacks, we will consider the diagram 
\begin{equation} \label{good}
\begin{tikzcd}
S \times_M S \arrow[r] \arrow[d] & S \arrow[r, equal] \arrow[d] & S \arrow[d] \\
S \arrow[r]                      & {[S/\awy]} \arrow[r, "F"]                       & M          
\end{tikzcd}
\end{equation}
and show that $S \cong S \times_M [S/\awy]$, from which it follows (by smooth base change) that $F: [S/\awy] \to M$ is an isomorphism. From Lemma \ref{technicallemma}, we know that $S \cong S \times_M [S/\awy]$ if and only if $S \times_M S$ is isomorphic to $S \times_{[S/\awy]} S$.

The diagram \eqref{good} induces a morphism
\[\phi: S \times_M S \to S \times_{[S/\awy]} S  \] 
of superstacks, which are actually represented by a superschemes by Corollary \ref{diagonalrep} and the fact that quotient stacks are algebraic. We may treat $\phi$ as a morphism of $S$-superschemes by observing that it respects the smooth second projections $p_2: S \times_M S \to S$ and $p_2': S \times_{[S/\awy]} S \to S$, making the following diagram commute:
\begin{equation*}
\begin{tikzcd}
S \times_M S \arrow[rr, "\phi"] \arrow[rd, "p_2"'] &   & S \times_{[S/\awy]} S. \arrow[ld, "p_2'"] \\                                                   & S &                                
\end{tikzcd}
\end{equation*} 

Since $p_1$ and $p'_1$ are flat morphisms to the Artinian base $S = \Spec B$, by Nakayama's lemma it is enough to check that the pullback of $\phi$ via $\Spec k \to S$ is an isomorphism.
The pullback of $\phi$ to $\on{Spec} k$  is a morphism
\[
S \times_M \Spec k \xrightarrow{\phi(k)} S \times_{[S/\awy]} \Spec k .
\]

The product $S \times_M \on{Spec} k$ is the functor sending a superscheme $T$ over $k$ to the set  $\{ f: T \to S , \; \phi: Z \times_{S,f} T \lgr \wy \times T\} $. 
This set is a torsor over $\awy(T)$ and, therefore, $S \times_M \on{Spec} k$ is an $\awy$-torsor.

On the other hand, $S \times_{[S/\awy]} \Spec k \cong \awy$ is also an $\awy$-torsor and the morphism $\phi(k)$ is an $\awy$-equivariant morphism of $\awy$-torsors, because the action of $\awy$ on both sides comes from the action on the first factor $S$ and $\phi(k)$ is induced by $\id: S \to S$.
Thus, $\phi(k)$ and therefore $\phi$ are isomorphisms.
\end{proof}

\subsection{Construction of the moduli superstack $\mnr$ of SUSY curves}
\label{susystructure} 

There is a natural map $\mnr \to M(1-\nr/2)$ forgetting the SUSY
structure. The fiber product $\mc{Y}$ in the diagram
    \begin{equation} \label{main}
\begin{tikzcd}
\mc{Y} \arrow[r] \arrow[d] & S \arrow[d] \\
                    \mnr \arrow[r]                    & M(1-\nr/2)                     
\end{tikzcd}
\end{equation} 
is a functor parameterizing SUSY structures on the universal
supercurve over $S$. More precisely, $\mc{Y}$ is the functor taking $T
\to S$ to the set of $\nr$-punctured SUSY structures on $Z \times_{S}
T$.  Our first goal in this section is to prove that $\mc{Y}$ is
represented by a superscheme. To do this, we will need to get a good
handle on those line bundles on $Z$ which encode an $\nr$-punctured
SUSY structure on $Z$. Henceforth, we will just refer to an
$\nr$-punctured SUSY structure as a SUSY structure.

\subsubsection*{Framed pre-SUSY structures}

Suppose that $\mc{L}$ is a line bundle on $Z$ realizing a SUSY
structure on $Z$: $\mc{L}$ is the subbundle of $\Omega^1_{Z/S}$ of
1-forms vanishing on the structure distribution $\mc{D}$, see
Definition \ref{SRSdef}. Then for the restriction of $\mc{L}$ to $\wy$,
we must have $\mc{L} \vert_{\wy} \cong \mc{O}_{\wy}(-2)$, see
Corollary \ref{-2}.  In particular, $\mc{L}$ may be viewed as a
deformation of $\mc{O}_{\wy}(-2)$, up to isomorphism.

From the gluing description of $Z$, we see that the standard gluing
description of the line bundle $\mc{O}(n)$ (transition function $z^n$
over the two standard overlapping charts) can still be used to define
a line bundle over $Z$, which we denote $\mc{O}_Z(n)$.

Since $\mc{O}_{\wy}(-2)$ has a unique, up to isomorphism, deformation
to $Z$ (Lemma \ref{otwodeformsuniquely}),
it follows that $\mc{L}$ must be isomorphic $\mc{O}_Z(-2)$.  It is,
therefore, a necessary condition
that for $\mc{L}$ to realize a SUSY structure on $Z$, it must be
isomorphic to $ \mc{O}_{Z}(-2)$.

From this necessary condition on $\mc{L}$, we are lead to the
following notion of a \emph{framed pre-SUSY structure} on $Z$.

\begin{definition}[Framed Pre-SUSY Structure]
A \emph{framed pre-SUSY structure} on $Z$ is a locally free, rank
$(0|1)$ quotient $q: \Omega^1_{Z/S} \twoheadrightarrow \mc{G}$
together with an isomorphism $\mc{L} := \on{ker}(q) \lgr \mc{O}_Z(-2)$
of line bundles on $Z$.
\end{definition}

Henceforth, let
\[
n := \nr
\]
denote the number of Ramond punctures, which is assumed to be even and
at least $4$.  We would like to start with computing the space
$\Hom (\mc{O}_Z(-2), \Omega_{Z/S}^1) =  H^0(Z, \Omega_{Z/S}^1(2))$ of
global sections of the twisted cotangent bundle $\Omega_{Z/S}^1(2)$.

\begin{proposition}
\label{7-18}
For the weighted projective superspace $\wy$, we have $\dim H^0(\wy, \Omega_{\wy}^1(2)) = n+2 \, | \, n+2$ and $H^1(\wy, \Omega_{\wy}^1(2)) = 0$.
\end{proposition}

\begin{proof}
Relating the sheaf of K\"ahler differentials on $\wy$ with that on the corresponding affine superspace $\mathbb{A}_k^{2|1} = \Spec k[u,v,\theta]$, where $k[u,v,\theta]$ is the homogeneous coordinate ring of $\wy$, we get a short exact sequence
\begin{equation}
    \label{SES}
0 \to \Omega_{\wy}^1 \to \mc{O}(-1) du \oplus \mc{O}(-1) dv \oplus \mc{O}(n/2 -1) d\theta \xrightarrow{p} \mc{O} \to 0,
\end{equation}
where $du$, $dv$, and $d\theta$ are mere placeholders for the components of the direct sum, also carrying certain parity: $du$ and $dv$ is even and $d\theta$ is odd. The morphism $p$ is given by
\begin{align}
\nonumber
du & \mapsto u \in H^0(\wy, \mc{O}(1)),\\
\label{p}
dv & \mapsto v \in H^0(\wy, \mc{O}(1)),\\
\nonumber
d\theta & \mapsto \theta \in H^0(\wy, \mc{O}(1-n/2)).
\end{align}
The short exact sequence \eqref{SES} immediately identifies the \emph{dualizing sheaf}, given by the Berezinian bundle $\Ber_{\wy} = \Ber \Omega_{\wy}^1$, as $\Ber (\mc{O}(-1)du \oplus \mc{O}(-1)dv \oplus \mc{O}(n/2 -1)d\theta) = \Pi \mc{O}(-n/2-1)$. Since we are interested in the twisted cotangent sheaf, let us twist \eqref{SES} by $\mc{O}(2)$ to get a short exact sequence
\begin{equation}
    \label{SES2}
0 \to \Omega_{\wy}^1(2) \to \mc{O}(1) du \oplus \mc{O}(1) dv \oplus \mc{O}(n/2 +1) d\theta \xrightarrow{p} \mc{O}(2) \to 0,
\end{equation}
where we slightly abused notation and denoted the right map by $p$, taking into account that it is given by the same formulas \eqref{p}. By super Serre duality \cite{voronov-manin-penkov}, $H^1(\mc{O}(1)) = H^0(\Pi \mc{O}(-n/2-2))^*$. This group vanishes, as the elements of the homogeneous coordinate ring have degree at least $1 - n/2$. Likewise, $H^1(\Pi \mc{O}(n/2+1)) = H^0(\mc{O}(-n-2))^* = 0$ and $H^1(\mc{O}(2)) = H^0(\Pi \mc{O}(-n/2-3))^* = 0$.
Thus, the long exact sequence associated with \eqref{SES2} looks as follows:
\begin{multline*}
0 \to H^0(\Omega_{\wy}^1(2)) \to H^0(\mc{O}(1)) du \oplus H^0(\mc{O}(1)) dv \oplus H^0(\mc{O}(n/2 + 1)) d\theta \\
\xrightarrow{H^0(p)} H^0(\mc{O}(2))
\to H^1(\Omega_{\wy}^1(2)) \to 0.
\end{multline*}
Now we can easily identify $H^0(\mc{O}(1))$ , $H^0(\mc{O}(2))$, and $H^0(\mc{O}(n/2+1))$ as well as the map $H^0(p)$:
\begin{align*}
  H^0(\mc{O}(1)) & = \left\{ a_0 u + a_1 v + \sum_{i=0}^{n/2} b_i u^{n/2-i}  v^i \theta \right\},\\
 H^0(\mc{O}(n/2+1)) & = \left\{\sum_{i=0}^{n/2+1} c_i u^{n/2+1-i} v^i + \sum_{i=0}^{n} d_i u^{n-i}  v^i \theta \right\},\\
 H^0(\mc{O}(2)) & = \left\{e_0 u^2 + e_1 uv + e_2 v^2 + \sum_{i=0}^{n/2+1} f_i u^{n/2+1-i}  v^i \theta \right\},
\end{align*}
where all the coefficients $a_i, b_i, c_i, d_i, e_i,$ and $f_i$ are elements of the ground field $k$. From the explicit formulas \eqref{p}. it is clear that $H^0(p)$ is surjective, whence
\[
H^1(\Omega_{\wy}^1(2)) = 0
\]
and
\[
\dim H^0(\Omega_{\wy}^1(2)) = n+2 \, | \, n+2. \qedhere
\]
\end{proof}

Computing the kernel of $H^0(p)$ in the proof of the proposition, we see that the forms
\begin{gather}
\nonumber
u dv - v du,\\
\nonumber
u^n \theta d\theta, u^{n-1} v \theta d\theta,  \dots, v^n \theta d\theta,\\
\label{basis}
u^{n/2-1}\theta (udv-vdu), u^{n/2-2} v\theta (udv-vdu), \dots, v^{n/2-1} \theta (udv-vdu),\\
\nonumber
u^{n/2} (u d\theta -\theta du), u^{n/2-1}v (ud\theta - \theta du), \dots, v^{n/2} (u d\theta - \theta du),\\
\nonumber
v^{n/2}(vd\theta - \theta dv)
\end{gather}
make up a basis of the super vector space $H^0(\Omega_{\wy}^1(2))$, where the first $n+2$ elements are even and the last $n+2$ elements are odd.

\begin{corollary}
\label{H0}
The $B$-module $H^0(Z, \Omega_{Z/S}^1(2))$ is free of rank $(n+2 \, | \, n+2)$. Moreover, it has a basis whose reduction to $B_{\bos} = k$ is given by formulas \eqref{basis}.
\end{corollary}

\begin{proof}
Since $H^1(\Omega_{\wy}^1(2)) = 0$, the $B$-module $H^0(\Omega_{Z/S}^1(2))$ is locally free. But the ring $B$ is local, so this $B$-module is actually free. The statement about the basis follows from Nakayama's lemma.
\end{proof}

We view 
\[
\mathbb{H}^0(\Omega_{Z/S}^1(2)) := \on{Spec} \mc{S}(H^0(\Omega_{Z/S}^1(2))^\vee)
\]
as a superscheme over $S$, a relative affine superspace of relative dimension $(n+2 \, | \, n+2)$, to be more precise.
If we choose a basis 
\[
\{x_1, x_2, \dots, x_{n+2}\, | \,  \xi_1, \dots, \xi_{n+2}\}
\]
for $H^0(\Omega_{Z/S}^1(2))^\vee$
dual to the basis \eqref{basis}, then we can identify this affine superspace with the standard one:
\[
\mathbb{H}^0(\Omega_{Z/S}^1(2)) = \on{Spec} \mc{S}(H^0(\Omega_{Z/S}^1(2))^\vee) \cong \on{Spec} B[x_1, \dots, x_{n+2}\, |\,  \xi_1, \dots, \xi_{n+2}].
\]

\begin{sloppypar}
The image of a framed pre-SUSY structure $u: \mc{O}(-2) \hookrightarrow \Omega_{Z/S}^1$ under the isomorphism $\on{Hom} (\mc{O}(-2), \Omega_{Z/S}^1) \xrightarrow{\sim} H^0(\Omega_{Z/S}^1(2))$  is a global section  $\mu \in H^0(\Omega_{Z/S}^1(2))^+$ whose coefficient by $udv-vdu$ is invertible. 
\end{sloppypar}

We, therefore, take the first coordinate $x_1$ to be invertible and identify the open affine sub-superscheme
\begin{equation}
\label{W}
W:= \on{Spec} B[x_1, x_1^{-1}, \dots, x_{n+2}\, | \, \xi_1, \dots, \xi_{n+2}] \subset \mathbb{H}^0(\Omega_{Z/S}^1(2))
\end{equation}  
as the superscheme of isomorphism classes of framed pre-SUSY structures on $Z$. 

\subsubsection*{Framed SUSY structures}
\label{zassrs}

By Corollary \ref{H0}, the space of framed pre-SUSY structures on $Z$ is represented by the superscheme $W$, which is an open subsuperscheme of the affine superspace of global sections of $\Omega^1_{Z/S}(2)$. Here we describe explicitly an open subsuperscheme of $W$ corresponding to the framed SUSY structures on $Z$. These are those framed pre-SUSY structures whose Ramond divisor $\mc{R}$ is unramified over the base, which means that the bosonization $\mc{R}_{\bos}$ is unramified over the bosonization of the base. This is equivalent to saying that, for each closed point $w$ of the base, the fiber $\mc{R}_{\bos}|_w$ over $w$ is a \emph{reduced divisor}, \emph{i.e}., , all its coefficients are equal to 1: $\mc{R}_{\bos}|_w = r_1 + \dots + r_{n}$ with $r_1, \dots, r_{n}$ being distinct closed points of $\wy_{\bos} = \mathbb{P}^1$.

By construction, $Z \times_S W$ has a canonical framed pre-SUSY structure whose restriction to $\wy \times W$ over $S_{\bos}$ is given by the 1-form
\begin{multline}
\label{canonicalpresusy}
x_1 (udv - vdu) + p(u,v) \theta d\theta
+ q(u,v) \theta (udv - v du) \\+ r(u,v) (u d\theta -\theta du) + \xi_{n+2} v^{n/2} (v d\theta -\theta dv)
\end{multline}
of homogeneity degree two, where
\begin{align}
\label{poly}
p(u,v) & := x_2 u^{n}+ x_3 u^{n-1} v + \dots + x_{n+2} v^{n},\\
\nonumber
q(u,v) & := \xi_1 u^{n/2-1} + \xi_2 u^{n/2-2} v + \dots + \xi_{n/2} v^{n/2-1},\\
\nonumber
r(u,v) & := \xi_{n/2+1} u^{n/2} + \xi_{n/2+2} u^{n/2-1} v + \dots + \xi_{n+1} v^{n/2}.
\end{align}
The corresponding framed pre-SUSY structure on $Z \times_S W$ defines a Ramond divisor $\mc{R}$ on $Z \times_S W$. Since we are interested in the unramification condition on $\mc{R}_{\bos}$, we can reduce the 1-form \eqref{canonicalpresusy} further to $\wy \times W_{\bos}$ to get a $1$-form
\begin{equation}
\label{canonicalpresusybos}
\varpi := x_1 (udv - v du) + p(u,v) \theta d \theta
\end{equation}
on $\wy$.

\begin{lemma}
The Ramond divisor $\mc{R}_{\bos}$ on $\wy \times W_{\bos}$ is given by the equation $p(u,v) = 0$.
\end{lemma}
\begin{proof}
Indeed, the form $\varpi$ restricts to
\begin{equation}
\label{varpi}
x_ 1 dz + p(1,z) \zeta d\zeta
\end{equation}
over the chart $U = \{[u:v \, | \, \theta] | u \neq 0 \} \subset \wy$ with standard affine coordinates $(z \, | \, \zeta) = ( \frac{v}{u} \, | \, \frac{\theta}{u})$.
The distribution $\mc{D} = \ker \varpi$ corresponding to \eqref{canonicalpresusybos} is locally generated by the vector field
\[
x_1 \delp{\zeta} - p(1,z) \zeta \delp{z},
\]
whose supercommutator with itself is
\[
- 2 x_1 p(1,z) \delp{z}.
\]
The Ramond divisor $\mc{R}_{\bos} \vert_{U \times W_{\bos}}$ is defined by the function $p(1,z)$. Since the form \eqref{canonicalpresusybos} is symmetric in $u$ and $v$ up to signs, the same argument will work in the chart $V = \{v \ne 0\}$ with affine coordinates $(w \, | \, \chi) = ( \frac{u}{v} \, | \, \frac{\theta}{v})$. Indeed,  the Ramond divisor $\mc{R}_{\bos} \vert_{V \times W_{\bos}}$ is  defined by the function $p(w,1)$. Getting back to the homogeneous coordinates $[u:v \, | \, \theta]$ on $\wy$, we see that the Ramond divisor $\mc{R}_{\bos}$ on $\wy \times W_{\bos}$ is defined by the homogeneous equation $p(u,v) = 0$.
\end{proof}

The unramification condition on $\mc{R}_{\bos}$ is equivalent to $p(u,v)$ having no multiple factors. Thus, the ramification locus on $W_{\bos}$ is the closed set given by homogeneous discriminant $\on{Disc}^h (p)$ of \eqref{poly}: $\on{Disc}^h (p) = 0$. Accordingly, the unramification locus is the open set $\on{Disc}^h (p(u,v)) \ne 0$, which defines the open subsuperscheme
\begin{equation}
\label{Y}
Y = \Spec B[x_1, x_1^{-1}, x_2,  \dots, x_{\nr+2}, (\on{Disc}^h (p))^{-1} \, | \, \xi_1, \dots, \xi_{\nr+2}] \subset W,
\end{equation}
the space of framed SUSY structures on $Z/S$.

\subsubsection*{Unframed SUSY structures}

Now we are ready to ``unframe'' SUSY structures on $Z$. For that, we need to extend the group $k$-superscheme $\Gamma^*$ used earlier to define $\awy$ and acting on sheaves on $\wy$ and their cohomology to a group $S$-superscheme $\Gamma^*_Z$ acting on sheaves on $Z$ and their cohomology as well as on $Y$. The group superscheme $\Gamma^*_Z$ over $S$ represents the functor
\begin{align}
    \Gamma^*_Z: \sS/S & \to \on{Group}, \\
   \nonumber T & \mapsto \Gamma((\mc{O}_{Z \times_S T}^+)^*).
\end{align}
This group superscheme is based on an open subsuperscheme in the relative affine superspace $\mathbb{H}^0 (Z, \mc{O}_Z)$ (or, rather, the one corresponding to the super vector bundle $\pi_* \mc{O}_Z$ over $S$, where $\pi: Z \to S$ is the structure projection). The sheaf $\pi_* \mc{O}_Z$ is locally free of rank equal to $\dim H^0 (\wy, \mc{O}_{\wy}) = (1 \,| \, \nr/2)$, because of the vanishing $H^1 (\wy, \mc{O}_{\wy}) = 0$ and the cohomology and base change argument, as in the proof of Proposition \ref{zisprojective}.2. The open subsuperscheme is defined by the complement to the zero section of $H^0(\wy, \mc{O}_{\wy}^+)$. Thus, $\Gamma^*_Z$ is a group superscheme over $S$ of relative dimension $\dim \Gamma^*_Z/S = (1 \,| \, \nr/2)$.

\begin{theorem}
\label{quotientequal}
The $S$-superscheme $Y/\Gamma^*_Z$ of relative dimension $\dim Y/S - \dim \Gamma^*_Z/S = (\nr+1 \, | \, \nr/2+2)$ represents the set of SUSY structures on $Z$. In particular, the fiber product $\mc{Y}$ in diagram \eqref{main} is represented by $Y/\Gamma^*_Z$.
\end{theorem}

\begin{proof} Recall that a framing on a SUSY structure on $Z$ was defined by fixing an isomorphism $c: \mc{L} \lgr \mc{O}_{Z}(-2)$.
The superscheme of isomorphisms $c$
is a $\Gamma^*_Z$-torsor.
We claim that the group superscheme $\Gamma^*_Z$ acts freely on the superscheme $Y$. Observe that the group superscheme $\Gamma^*_Z$ acts on
the relative affine superspace $\mathbb{H}^0(Z, \Omega_{Z/S}^1(2))$ by multiplication. Therefore, $\Gamma^*_Z$ also acts on the open subsuperschemes $W$, see \eqref{W}, and on the open subsuperscheme $Y$, see \eqref{Y}. And this action is free, because it is free on the element $udv - v du \in H^0(\wy, \Omega_{\wy}^1(2))$, see \eqref{basis}, corresponding to the coordinate $x_1$ on $\mathbb{H}^0(\wy, \Omega_{\wy}^1(2))$, which is invertible on $W$ by construction.
\end{proof}

Now we can replace $\mc{Y}$ with $Y/\Gamma^*_Z$ in \eqref{main}, and it 
follows that the morphism $Y/\Gamma^*_Z \to \mnr$ is smooth, surjective, and representable. In particular, we have

\begin{theorem} \label{maintheorem} $\mnr \cong [(Y/\Gamma^*_Z)/\awy]$. 
\end{theorem}

\begin{proof}
The action of $\awy$ on $S$ lifts naturally to the relative affine superspace $\mathbb{H}^0(Z, \Omega_{Z/S}^1(2))$ and the open subsuperscheme $Y \subset \mathbb{H}^0(Z, \Omega_{Z/S}^1(2))$, preserves the orbits of the action of $\Gamma^*_Z$, and, thereby, 
descends to an action on $Y/\Gamma^*_Z$.

Moreover, for every superscheme $T$, the fiber product of the diagram
\begin{center}
\begin{tikzcd}
& Y/\Gamma^*_Z \arrow[r] \arrow[d] & S \arrow[d]        \\
  T \arrow[r]  & \mnr \arrow[r] & {M(1-\nr/2)=[S/\awy]}
\end{tikzcd}
\end{center}
is isomorphic to an $\awy \times T$-torsor; and so
$\mnr \cong [(Y/\Gamma^*_Z)/\awy]$.
\end{proof}

\subsubsection*{Automorphisms of $\nr$-punctured SUSY curves} \label{automorphisms}

We would also like to show that $\mnr$ is a Deligne-Mumford superstack. To do that, we need to show that $\awy(\on{Spec} k)$ acts  on $(Y/\Gamma^*_Z)(\on{Spec}k)$ with finite stabilizers. The set $(Y/\Gamma^*_Z)(\on{Spec} k)$ is in bijection with the set of SUSY structures on $\wy$. The SUSY structure on $\wy$ are locally, in the chart $U$, generated by the $1$-forms $dz + p(z) \zeta d \zeta$, see \eqref{varpi}, and $\awy(\on{Spec} k)$ naturally acts on these via pullback. 
To prove that $\awy(\on{Spec} k)$ acts with finite stabilizers, we will show that the only nontrivial automorphism of $\wy$ with SUSY structure locally generated by the $1$-form $dz + p(z) \zeta d \zeta$ is the canonical automorphism sending $z \mapsto z$ and $\zeta \mapsto - \zeta$.

A \emph{superconformal vector field} on a super Riemann surface $\Sigma$ based on $\wy$ is a vector field $\mathfrak{X}$ which preserves the SUSY structure, in the sense that  $[\mc{D}, \mathfrak{X}] \subset \mc{D}$. We denote by $\mc{A}_{\Sigma}$ the sheaf of superconformal vector fields on $\Sigma$.
Following \cite[Section 4.2.1]{witten2012notes}, we can compute the space of local sections of $\mc{A}_{\Sigma}$. In our standard chart $U$, the odd component  $H^0(U, \mc{A}|_{U})^-$ consists of the odd vector fields
\begin{align}
\label{oddsuperconformalvectorfield}
f(z) \left( \delp{\zeta} - p(z) \zeta \delp{z} \right), \end{align}
where $p(z)$ denotes the defining local section of the Ramond divisor $\mc{R}$ and $f(z)$ is an arbitrary even function on $U \subset \wy$.
The even component $H^0(U, \mc{A}|_{U})^+$ consists of the even vector fields
\begin{align}
\label{evensuperconformalvectorfield}
p(z) \left( g(z) \delp{z} + \frac{g'(z)}{2} \zeta \delp{\zeta} \right),
\end{align}
where $g(z)$ is an even function on $U$.

\begin{theorem}\label{infaut} If $\nr \ge 4$, $H^0(\mc{A}_{\Sigma})=0$. 

\end{theorem} 

\begin{proof} We will show that the local sections of $\mc{A}$ described in \eqref{oddsuperconformalvectorfield} and \eqref{evensuperconformalvectorfield} do not extend to give global sections of $\mc{A}$. This will prove the theorem, since the restriction of any global section of $\mc{A}$ to $U$ must be equal to either  \eqref{oddsuperconformalvectorfield} or \eqref{evensuperconformalvectorfield}.

A \v{C}ech cohomology computation shows that the vector fields
\begin{equation}
    \label{H0T}
 \left\{ \delp{z}, z \delp{z}, z^2 \delp{z} + z \zeta \delp{\zeta}, \zeta \delp{\zeta} \, \left\lvert \; \zeta \delp{z}, z \zeta \delp{z}, \dots, z^{\nr/2+1} \zeta \delp{z} \right. \right\}
 \end{equation}
on $U$ extend to a basis of $H^0(\wy, \mc{T}_{\wy})$. This means that any odd (even) vector field on $\wy$ restricts to $U$ as a linear combination
\begin{align}
\label{oddvectorfield}
\mathfrak{X}^- & = \sum_{i=0}^{\nr/2+1} b_i z^i \zeta \delp{z}, \qquad \text{or} \\
\label{evenvectorfield}
\mathfrak{X}^+ & = \left(a_0 + a_1 z + a_2 z^2\right) \delp{z} + \left(a_3
+ a_2z \right)\zeta \delp{\zeta},
\end{align}
respectively.

Comparing \eqref{oddvectorfield} to \eqref{oddsuperconformalvectorfield}, observe that there do \emph{not} exist any global odd superconformal vector fields on $\wy$.
Comparing \eqref{evenvectorfield} to \eqref{evensuperconformalvectorfield}, we see that $\nr = \deg p(z) \le 2$, which 
contradicts the assumption that $\nr \ge 4$, and thus $\mathfrak{X}^+=0$. 
\end{proof} 

The above theorem implies that the group of automorphisms of a SUSY-structure on $\wy$ is discrete. We can actually improve this statement, as follows.

\begin{theorem}
\label{finitestabilizer}
For $\Sigma$ over $\on{Spec} k$, $\on{Aut}(\Sigma) \cong \mathbb{Z}/2\mathbb{Z}$.
\end{theorem} 

\begin{proof}  An automorphism of the homogeneous coordinate ring of $\wy$, \emph{i.e}., the $\Z$-graded $k$-su\-per\-al\-ge\-bra $A = k[u, v \, | \,\theta]$ with $u,v$ even of degree 1 and $\theta$ odd of degree $1-\nr/2$, sends
\begin{align} 
\nonumber
u & \mapsto au + bv  \\ 
\label{aut}
v & \mapsto cu + dv \\ 
\nonumber \theta & \mapsto e \theta 
\end{align} 
where $a,b,c,d \in k$, $e \in k^*$ and $ad - bc \neq 0$. The above map restricts on the chart $U = \on{Spec} k[z,\zeta]$ to the map
\begin{align*}
z & \mapsto \frac{c + dz}{a + bz},\\
\zeta & \mapsto e \zeta(a+ bz)^{\nr/2-1}.
\end{align*}

Let $dz + p(z) \zeta d \zeta$ denote a generating section of an $\nr$-punctured SUSY structure on $\wy$, where $p(z)$ is the local defining equation of the Ramond divisor, as in \eqref{varpi}.

An automorphism of $\awy$ which preserves the SUSY structure must preserve the Ramond divisor, in particular. The bosonic reduction a SUSY-automorphism must preserve the bosonic reduction of the Ramond divisor, which is
$\nr \ge 4$ marked points on $\pr$. The only automorphism preserving $\nr \ge 4$ marked points on $\pr$ is the identity. This immediately implies that $a=d$
and $b=c=0$ in \eqref{aut}.

Any such automorphism preserves the charts $U$ and $V$ and acts as
\[ dz + p(z) \zeta d \zeta \mapsto dz + e^2 a^{\nr-2} p(z) \zeta d \zeta \]
on the generating 1-form of the SUSY structure. 
Thus, $e a^{\nr/2-1} = \pm 1$. In particular, any SUSY-automorphism of $\wy$ must send $z \mapsto z$ and $\zeta \mapsto \pm \zeta$.
\end{proof}

\begin{corollary} \label{finat}
$\mnr$ is a Deligne-Mumford superstack of dimension $\dim Y/\Gamma^*_Z - \dim \awy = (\nr-3\, |\, \nr/2-2)$. 
\end{corollary} 

\bibliographystyle{amsalpha}
\bibliography{versionfive}

\end{document}